\documentclass[11pt]{article}

\usepackage{mathptmx}
\usepackage{preprint-layout}
\usepackage{preprint-notation}
\usepackage{a4wide}
\usepackage{float}
\usepackage{placeins}
\usepackage{url}

\graphicspath{{figures/}}

\title{A Barrier-Regularized Symmetric Nitsche Method for the Signorini Problem}

\author{
Peter Hansbo \qquad Mats G. Larson
}

\date{}

\begin{document}

\maketitle

\begin{abstract}
We introduce and analyze a barrier-regularized symmetric Nitsche method for
the scalar Signorini problem. Applying a logarithmic barrier to the
nonnegative slack variable in an augmented Lagrangian and then eliminating
that variable yields a smooth positive-part operator and the perturbed
complementarity relation on a primal-dual central path, with barrier parameter
\(\mu=\gamma s\). For every \(s>0\), the discrete problem is continuously
differentiable, uniquely solvable, and has a symmetric positive definite
Newton matrix when the Nitsche parameter is sufficiently large. Under a mild
barrier-feasibility condition, the continuous logarithmic energy has a unique
minimizer \(u_\mu\) for every \(\mu>0\), its gap is positive almost
everywhere, and
\(\|u-u_\mu\|_{H^1(\Omega)}\lesssim\mu^{1/2}\). Under the Sobolev regularity
used for finite element approximation, this minimizer satisfies the
\(L^2\) central-path boundary law. If the obstacle is locally the trace of
an \(H^2\)-function, we also prove a uniform local \(H^2\) bound in the
interior of each planar contact face, on neighborhoods that may contain a
free-boundary point of the limiting Signorini solution. The method is exactly
consistent and quasi-optimal relative to
\(u_\mu\), with constants independent of \(s\). If
\(u_\mu\) is uniformly bounded in
\(H^r(\Omega)\), \(3/2<r\leq k+1\), then
\(\|u-u_h\|_{H^1(\Omega)}\lesssim
h^{r-1}\|u_\mu\|_{H^r(\Omega)}+\mu^{1/2}\); hence
the sufficient balance \(s\lesssim h^{2r-1}\) preserves the available
energy-norm rate. We also derive rates for discrete penetration and the
complementarity residual. Numerical
experiments with \(P_1\) and \(P_2\) elements on structured and unstructured
meshes support the predicted rates and the local Newton theory, while showing
that a rate-preserving smoothing may still resolve the contact set
poorly.

\end{abstract}




\section{Introduction}
\label{sec:introduction}

The Signorini problem is a canonical model of unilateral contact: the
displacement is constrained by an obstacle on part of the boundary, while the
contact region is not known in advance. Its complementarity conditions lead
to a variational inequality and, after finite element discretization, to a
nonsmooth nonlinear system. This scalar setting isolates the analytical and
algorithmic issues that also arise in elastic contact.

Nitsche's method enforces the contact conditions without introducing an
independent discrete multiplier. In the standard symmetric formulation, the
conditions are encoded by a positive-part operator. The resulting method is
consistent and accurate, but the positive part is not differentiable at the
origin and is therefore usually paired with semismooth Newton or primal-dual
active-set solvers. We study instead the smooth approximation
\begin{equation}
\varphi_{+,s}(w)=\frac{w}{2}+\sqrt{\frac{w^2}{4}+s},
\quad s>0 \label{eq:introduction-regularized-positive-part}
\end{equation}
which converges uniformly to \([w]_+\) as \(s\to0\). The central question is
not only how to solve the smoothed equations, but also how the smoothing
parameter should scale with the mesh size so that regularization does not
degrade the finite element error.

\paragraph{Contributions.}
The contribution has four parts.
\begin{enumerate}
\item \emph{Barrier derivation.} We show that
\eqref{eq:introduction-regularized-positive-part} is obtained exactly by
placing a logarithmic barrier on the nonnegative slack variable in the
augmented-Lagrangian formulation and then eliminating that variable. The
regularized contact law is consequently the perturbed complementarity
condition on a primal-dual central path, with
\begin{equation}
\mu=\gamma s \label{eq:introduction-barrier-smoothing-relation}
\end{equation}
where \(\mu\) is the barrier parameter and \(\gamma\) is the Nitsche
parameter.

\item \emph{Discrete solvability and Newton analysis.} For every \(s>0\), the
discrete residual is continuously differentiable. We derive its Newton
linearization and a Taylor remainder estimate, prove unique solvability, and
show that the Newton matrix is symmetric positive definite at every iterate
when the Nitsche parameter is sufficiently large. For fixed mesh and positive
\(s\), the local Newton rate is quadratic; the sufficient convergence radius
provided by the analysis deteriorates as \(s\to0\) and is not uniform in the
mesh size.

\item \emph{Continuous and discrete central-path analysis.} Under a mild
barrier-feasibility condition, we prove that the continuous logarithmic energy
has a unique minimizer \(u_\mu\) for every \(\mu>0\), with gap positive almost
everywhere, and that
\begin{equation}
\|u-u_\mu\|_{H^1(\Omega)}\lesssim\mu^{1/2}
\label{eq:introduction-central-path-estimate}
\end{equation}
This part of the analysis does not require reciprocal-gap or normal-flux
regularity. When the obstacle is locally the trace of an \(H^2\)-function, the
monotonicity of the inverse-gap boundary law further gives a uniform local
\(H^2\)-bound in the interior of each planar contact face, on neighborhoods
that may contain a free-boundary point of the limiting Signorini solution. We
then prove that the Sobolev regularity used for finite element approximation
automatically yields the
\(L^2(\Gamma_C)\) central-path boundary law. Under a uniform Sobolev bound,
the method is exactly consistent and quasi-optimal relative to \(u_\mu\), with
constants independent of \(s\). If
\(3/2<r\leq k+1\) and \(u_\mu\) is uniformly bounded in \(H^r(\Omega)\), then
\begin{equation}
\|u-u_h\|_{H^1(\Omega)}
\lesssim h^{r-1}\|u_\mu\|_{H^r(\Omega)}+\mu^{1/2} \label{eq:introduction-main-error-estimate}
\end{equation}
and the sufficient rate-preserving balance is \(s\lesssim h^{2r-1}\). Under
full uniform regularity, \(r=k+1\), the sufficient choice
\(s\lesssim h^{2k+1}\) preserves the rate \(h^k\). A direct comparison with
the Signorini solution produces a consistency term and a condition stronger by
one power of \(h\). We also derive \(h^{r-1/2}\) bounds for the discrete
penetration and complementarity residual.

\item \emph{Numerical assessment.} Experiments with \(P_1\) and \(P_2\)
elements examine regularization error, Newton convergence, smooth and singular
exact solutions, contact pressure, the artificial gap, the Nitsche parameter,
and unstructured meshes. They support the predicted energy-norm rates and the
local Newton theory, and they show that a value of \(s\) adequate for the
energy error can still be too large for accurate \(L^2(\Omega)\) error or
contact-set resolution.
\end{enumerate}

\paragraph{Previous Work.}
The mathematical theory of unilateral contact as a variational inequality and
its finite element approximation originates in the classical works
\cite{Fa74,KiOd88}; see also the review \cite{Wo11} and the monograph
\cite{ChHiRe23}. Refined estimates for conforming \(P_1\) and \(P_2\)
elements were established in \cite{HiRe12,BeBe03}. Optimal energy-norm
estimates in two and three dimensions, without additional assumptions on the
unknown contact set, were obtained in \cite{DrHi15}. The
\(3/2\)-homogeneous profile used in our numerical tests reflects the optimal
\(C^{1,1/2}\) regularity at the Signorini boundary \cite{AtCa06}. Uniform
regularity estimates obtained through monotone penalized boundary-obstacle
approximations appear in \cite{MeNi05,Ja17}. Related peer-reviewed regularity
and free-boundary results for a two-penalty boundary problem are given in
\cite{DaJa21}; its boundary nonlinearity differs from the logarithmic law
studied here. In polygonal domains, the global Sobolev exponent is also
limited by angle-dependent singularities at transitions between Signorini and
Dirichlet or Neumann conditions \cite{ApNi20}.

Nitsche's method \cite{Ni71} was introduced for unilateral contact by Chouly
and Hild \cite{ChHi13}, who represented the Kuhn--Tucker conditions by a
single equation involving the positive part and proved optimal error
estimates. Symmetric and nonsymmetric variants were analyzed in
\cite{ChHiRe15}; we use the symmetric formulation. Related developments
include a penalty-free nonconforming method \cite{BuHaLa17} and a stabilized
mixed interpretation with a posteriori estimates \cite{GuStVi20}. The
augmented-Lagrangian viewpoint adopted here follows \cite{Ro74} and its
development for Nitsche-type contact methods in \cite{BuHaLa19,BuHaLa23}.
Di Pietro, Fontana, and Kazymyrenko \cite{DiFoKa22} regularized the contact
projection in a Nitsche formulation as part of an adaptive nonlinear solution
procedure and an a posteriori analysis. Their regularization separates
discretization, linearization, and algebraic errors and is tuned toward the
original nonsmooth discrete problem; it is not derived from the boundary
barrier or analyzed through the continuous central path considered here.

At the algebraic level, contact complementarity systems are commonly solved
by generalized or semismooth Newton methods and equivalent primal-dual
active-set strategies \cite{AlCu91,HiItKu02}. The operator in
\eqref{eq:introduction-regularized-positive-part} belongs to the
Chen--Harker--Kanzow--Smale family \cite{ChHa93,Sm87,Ka96}; smoothing methods
for nonsmooth optimization are surveyed in \cite{Ch12}. Kanno and Ohsaki
\cite{KaOh11} applied an implicit smoothing formulation to frictionless
contact in large deformations. Ben Gharbia et al.\ \cite{BeFeVoYo23} developed
an adaptive inexact smoothing Newton method for a finite-volume discretization
of a variational inequality, controlling smoothing, linearization, and
algebraic errors a posteriori.

Regularization and path following have also been studied in connection with
finite element discretization. Bank, Gill, and Marcia \cite{BaGiMa03} coupled
primal-dual interior methods to adaptive finite element approximations of
elliptic variational inequalities and addressed scaling across refinement.
Stadler \cite{St07} combined semismooth Newton iteration with
augmented-Lagrangian and path-following updates for elastic contact. That work
solves a sequence of regularized problems, whereas the present method fixes
\(s\) as a function of the mesh size and analyzes the resulting discretization
against a continuous central path. For domain obstacle problems, the proximal
Galerkin method \cite{KeSu26} uses entropy regularization and a latent-variable
formulation to preserve pointwise bounds.

Continuous logarithmic barriers and central paths for contact predate the
present work. Kloosterman et al.\ \cite{KlVaVaHu01} combined a barrier
algorithm with an augmented formulation for finite element frictionless
contact. Tanoh, Renard, and Noll \cite{TaReNo04} wrote the continuous
logarithmic boundary energy, its multiplier and slack-variable forms, and the
perturbed Kuhn--Tucker product before developing discrete primal and
primal--dual central paths. Temizer, Abdalla, and G\"urdal \cite{TeAbGu14}
developed primal and primal--dual interior-point methods with slack variables
for isogeometric mortar contact. Zhao et al.\ \cite{ZhChJi22} used a truncated
smooth barrier energy that deliberately retains an artificial gap, while Guo
et al.\ \cite{GuLiYaLiWa24} combined a barrier with an augmented Lagrangian for
elastodynamic contact. These works establish that logarithmic barriers,
slack-variable contact formulations, and continuous and discrete central paths
are not new ingredients here.

The contribution is instead their specific integration with symmetric Nitsche
contact. Exact elimination of the barrier-regularized slack variable produces
the smoothed positive-part operator used in the Nitsche boundary form. This
identity gives exact consistency against the continuous path, an
\(s\)-uniform quasi-optimal finite element estimate, and a
regularity-dependent sufficient balance between \(s\) and \(h\). Neither the
symmetric-Nitsche elimination nor this uniform discretization analysis appears
in the barrier-contact work above. General barrier and path-following theory is
described in \cite{FiMc68,Wr97,FoGiWr02}.

\paragraph{Outline.}
Sections~\ref{sec:model-problem}--\ref{sec:finite-element-method} present the
formulation, Sections~\ref{sec:newton-method}--\ref{sec:error-estimates} the
analysis, Section~\ref{sec:numerical-examples} the computations, and
Section~\ref{sec:concluding-remarks} the conclusions.


\section{The Model Problem}
\label{sec:model-problem}

Let \(\Omega\subset\mathbb{R}^d\), \(d\in\{2,3\}\), be a bounded polygonal
or polyhedral Lipschitz domain. We assume that its boundary is decomposed into
relatively open, disjoint parts \(\Gamma_D\) and \(\Gamma_C\) such that
\begin{align}
\lvert\Gamma_D\rvert&>0
\label{eq:positive-dirichlet-boundary-measure}\\
\overline{\Gamma}_D\cup\overline{\Gamma}_C=\partial\Omega
\label{eq:boundary-decomposition}
\end{align}
The extension to an additional Neumann part is immediate. We impose
homogeneous Dirichlet conditions on \(\Gamma_D\); nonhomogeneous data may be
handled by a standard lifting. Set
\begin{align}
V &\coloneqq \{v\in H^1(\Omega):v=0\ \text{on}\ \Gamma_D\}
\label{eq:energy-space}\\
a(v,w) &\coloneqq (\nabla v,\nabla w)_\Omega
\label{eq:bilinear-form}\\
l(w) &\coloneqq (f,w)_\Omega
\label{eq:load-functional}
\end{align}
where \(f\in L^2(\Omega)\). The assumption on \(\Gamma_D\) gives the
Poincar\'e inequality
\begin{equation}
c_P\|w\|_{H^1(\Omega)}^2\leq a(w,w)
\quad \forall w\in V
\label{eq:poincare-inequality}
\end{equation}

Let \(g\in H^{1/2}(\Gamma_C)\) denote the obstacle and suppose that the
closed convex set
\begin{equation}
K\coloneqq \{v\in V:v\leq g\ \text{a.e. on}\ \Gamma_C\}
\label{eq:admissible-set}
\end{equation}
is nonempty. The scalar Signorini problem is
\begin{equation}
u\in\operatorname*{argmin}_{v\in K}
\left(\frac12 a(v,v)-l(v)\right)
\label{eq:signorini-minimization}
\end{equation}
Equivalently, \(u\in K\) is the unique solution of the variational inequality
\begin{equation}
a(u,v-u)-l(v-u)\geq0
\quad \forall v\in K
\label{eq:signorini-variational-inequality}
\end{equation}


\paragraph{Strong Contact Notation.}
\label{par:strong-contact-notation}
In the natural weak formulation, the contact reaction is an element of the
dual trace space and need not be an \(L^2(\Gamma_C)\) function. Whenever the
solution is sufficiently regular that the normal flux
\(p=\sigma_n(u)=\nabla u\cdot\mathbf{n}\) admits a pointwise representative,
the contact conditions read
\begin{align}
\rho(u)\coloneqq u-g &\leq0
\quad \text{on }\Gamma_C
\label{eq:nonpenetration-condition}\\
p &\leq0
\quad \text{on }\Gamma_C
\label{eq:contact-pressure-sign}\\
\rho(u)p &=0
\quad \text{on }\Gamma_C
\label{eq:contact-complementarity}
\end{align}
Thus \(-p\) is the nonnegative contact pressure. For any \(\gamma>0\), these
conditions are equivalent pointwise to
\begin{equation}
p=-\gamma\,[\rho(u)-p/\gamma]_+
\quad \text{on }\Gamma_C
\label{eq:contact-positive-part-reformulation}
\end{equation}
where \([w]_+=\max(0,w)\). This strong reformulation motivates the Nitsche
method below. All uses of pointwise boundary products in the continuous
central-path analysis will be accompanied by explicit regularity and
integrability assumptions.


\section{The Finite Element Method}
\label{sec:finite-element-method}

Let \(\mcT_h\) be a shape-regular, quasiuniform, boundary-fitted partition of
\(\Omega\) with mesh parameter \(h\in(0,h_0]\). Let
\(V_h\subset V\cap C^0(\overline\Omega)\) be the conforming finite element
space of continuous piecewise polynomials of degree \(k\). On each contact
face, \(\sigma_n(v_h)\) denotes the normal derivative taken from its unique
adjacent element. We take
\(\pi_h:H^1(\Omega)\rightarrow V_h\) to be a boundary-preserving
Scott--Zhang quasi-interpolant \cite{ScZh90}. Its fractional-order local
stability and approximation properties follow by real interpolation; see
\cite{Ci13}. In particular,
\begin{equation}
\| v - \pi_h v \|_{H^m(T)} \lesssim h^{r - m} \| v \|_{H^{r}(N_h(T))} 
\label{eq:local-interpolation-estimate}
\end{equation}
for \(0\leq m\leq r\leq k+1\), where \(N_h(T)\) is the union of all elements
that share a node with \(T\).


\subsection{Augmented-Lagrangian Formulation}
\label{subsec:augmented-lagrangian-formulation}

The following augmented-Lagrangian calculation is an algebraic derivation of
the discrete boundary terms. All quantities are taken sufficiently regular
that the displayed \(L^2(\Gamma_C)\) products and pointwise eliminations are
meaningful. We introduce a slack variable \(z\in L^2_+(\Gamma_C)\), where
\(L^2_+(\Gamma_C)=\{q\in L^2(\Gamma_C):q\geq0\ \text{a.e.}\}\), and adopt
the normal-flux sign convention for the multiplier: at the solution,
\(\lambda=\sigma_n(u)\leq0\), while \(-\lambda\) is the nonnegative contact
pressure. The augmented Lagrangian is
\begin{equation}
\begin{aligned}
\mcL_A(v,\lambda,z) &= 
\frac12 \| \nabla v \|^2_\Omega - (f,v)_\Omega
\\
&\qquad 
- (((v-g) + z),\lambda)_{\Gamma_C} + \frac{\gamma}{2} \| (v - g) + z\|^2_{\Gamma_C}
\\
&= 
\frac12 \| \nabla v \|^2_\Omega - (f,v)_\Omega
\\
&\qquad 
+ \frac{\gamma}{2} \Big(  \| (v - g) + z - \lambda/\gamma\|^2_{\Gamma_C} - \|\lambda/\gamma\|^2_{\Gamma_C} \Big)
\end{aligned}
\end{equation}
where the second equality follows by completing the square. Taking the
constraint $z\geq 0$ into account, the minimum with respect to $z$
is attained in 
\begin{equation}
z = [- (v - g) + \lambda/\gamma]_+
\end{equation}
and using the identity $a+[-a]_+ = [a]_+$ we get 
\begin{align}
\mcL_A(v,\lambda) 
&= 
\frac12 \| \nabla v \|^2_\Omega - (f,v)_\Omega
+ \frac{\gamma}{2} \Big(  \| [(v - g) - \lambda/\gamma]_+\|^2_{\Gamma_C} - \|\lambda/\gamma\|^2_{\Gamma_C} \Big)
\end{align}
Next we eliminate the Lagrange multiplier using the substitution
$\lambda = \sigma_n(v)$, which gives the augmented Lagrangian
\begin{align}
\mcL_A(v) 
&= 
\frac12 \| \nabla v \|^2_\Omega - (f,v)_\Omega
+ \frac{\gamma}{2} \Big(  \| [(v - g) - \sigma_n(v)/\gamma]_+\|^2_{\Gamma_C} - \|\sigma_n(v)/\gamma\|^2_{\Gamma_C} \Big)
\end{align}
For \(\gamma=\gamma_0/h\), introduce the affine contact operator and its
differential,
\begin{align}
P(v)&\coloneqq (v-g)-\sigma_n(v)/\gamma
\label{eq:contact-operator}\\
DP(w)&\coloneqq w-\sigma_n(w)/\gamma
\label{eq:contact-operator-differential}
\end{align}
so that \(P(v+w)=P(v)+DP(w)\). The corresponding equilibrium equations take
the form
\begin{align}
u_h \in V_h:\qquad A_h(u_h,v) = l(v) \qquad \forall v\in V_h
\end{align}
where 
\begin{align}
A_h(v,w) &= a_h(v,w) + b_h(v,w)
\\
a_h(v,w) &= a(v,w) - \frac{1}{\gamma}(\sigma_n(v),\sigma_n(w))_{\Gamma_C}
\\
b_h(v,w) &= \gamma([P(v)]_+,DP(w))_{\Gamma_C}
\\
l(v) &= (f,v)_\Omega
\end{align}
Here \(a(v,w)=(\nabla v,\nabla w)_\Omega\).


\subsection{Barrier Regularization and Elimination}
\label{subsec:barrier-regularization-elimination}

We now return to the elimination of the slack variable. The variable \(z\)
represents the gap \(g-v\), with equality at a saddle point where the slack
constraint is satisfied. The nondifferentiability originates in the
constraint \(z\geq0\) becoming active on the contact zone. At the algebraic
level, we relax this constraint by a logarithmic barrier with parameter
\(\mu>0\),
\begin{align}
\mcL^\mu_A(v,\lambda,z) = \mcL_A(v,\lambda,z) - \mu (1,\ln z)_{\Gamma_C}
\end{align}
and carry out the elimination pointwise for \(z>0\). Since the scalar barrier
is smooth on \((0,\infty)\), its stationary point is characterized by
\begin{align}
D_z \mcL^\mu_A(v,\lambda,z) = -\frac{\mu}{z} - \lambda + \gamma\big((v-g) + z\big) = 0
\end{align}
or equivalently, after multiplication by $z/\gamma$ and with the notation
$w = (v-g) - \lambda/\gamma$,
\begin{align}
z^2 + w z - \mu/\gamma = 0
\end{align}
The unique positive root of this equation is
\begin{align}
z = -\frac{w}{2} + \sqrt{\frac{w^2}{4} + \frac{\mu}{\gamma}} = \varphi_{+,s}(-w),
\qquad s = \mu/\gamma
\end{align}
where we introduced the regularized max operator
\begin{align}
\boxed{\varphi_{+,s}(w) = w/2 + \sqrt{w^2/4 + s}}
\end{align}
with regularization parameter $s$. For $s=0$ we have precisely
\begin{align}
\varphi_{+,0}(w) = \max(0,w)
\end{align}
and thus the barrier-free elimination \(z=[-w]_+\) is recovered as
\(\mu\rightarrow0\). The regularized max operator is therefore precisely the
slack variable produced by the logarithmic barrier. The identity
\(a+[-a]_+=[a]_+\) carries over in the form
\begin{align}
w + \varphi_{+,s}(-w) = \varphi_{+,s}(w)
\end{align}
which follows immediately from the definition.

To see how the barrier enters the reduced equations, let \(z=z(v,\lambda)\)
denote the stationary root. In the chain rule for the reduced functional, the
term containing the derivative of \(z\) is multiplied by
\(D_z\mcL_A^\mu(v,\lambda,z)=0\). Moreover, the barrier has no explicit
dependence on \(v\). Its effect therefore enters entirely through the
stationary value of \(z\), and the reduced equilibrium equations have the same
form as before with \([\cdot]_+\) replaced by \(\varphi_{+,s}\).

We define the regularized contact form by
\begin{align}
b_{h,s}(v,w) &= \gamma(\varphi_{+,s}(P(v)),DP(w))_{\Gamma_C}
\end{align}
The full regularized form is
\begin{align}
A_{h,s}(v,w) &= a_h(v,w) + b_{h,s}(v,w)
\end{align}
and the discrete method reads
\begin{align}
u_h \in V_h:\qquad A_{h,s}(u_h,v) = l(v) \qquad \forall v\in V_h
\end{align}

\begin{rem}[Relation to interior point methods]
\label{rem:interior-point-relation}
The regularized contact condition
\begin{align}
p = -\gamma\varphi_{+,s}(\rho(v) - p/\gamma)
\end{align}
which replaces the strong contact reformulation
\eqref{eq:contact-positive-part-reformulation}, is
equivalent to
\begin{align}
\rho(v) < 0, \qquad p < 0, \qquad \rho(v)\, p = \gamma s = \mu
\end{align}
as follows by setting $a = -p/\gamma$ and $b = -\rho(v)$ and using
$\varphi_{+,s}(w)\varphi_{+,s}(-w) = s$. The complementarity condition
$\rho(v) p = 0$ is thus perturbed into $\rho(v) p = \mu$, which is precisely the
central path of a primal-dual interior point method; this equivalence is the
basis of the non-interior continuation methods of \cite{ChHa93,Ka96}.
For sufficiently regular functions satisfying this relation pointwise, the
gap \(g-v\) and contact pressure \(-p\) are therefore strictly positive for
\(s>0\). The relation is an interior-point relaxation rather than a penalty.
Section~\ref{sec:error-estimates} states the precise assumptions under which a
continuous central-path solution with these properties is used in the error
analysis and quantifies how fast \(\mu=\gamma s\) must tend to zero under mesh
refinement.
\end{rem}


\section{Newton's Method}
\label{sec:newton-method}

We first derive the Newton linearization and its Taylor remainder. We then
establish coercivity and unique solvability before collecting the consequences
for the linear algebra and local convergence of the iteration.


\subsection{Linearization and Remainder}
\label{subsec:newton-linearization-remainder}

Since $\varphi_{+,s} \in C^\infty(\IR)$ for $s > 0$, the form $A_{h,s}$ is
continuously differentiable in its first argument, and we have the expansion
\begin{align}
A_{h,s}(v+\delta, w) = A_{h,s}(v, w) + L_h(v;\delta,w) + R_h(v;\delta,w)
\label{eq:newton-expansion}
\end{align}
where the linearization of $A_{h,s}$ at $v$ takes the form
\begin{align}
L_h(v;\delta,w)
= a_h(\delta,w)
+ \gamma( \varphi'_{+,s}(P(v))\, DP(\delta), DP(w))_{\Gamma_C}
\label{eq:newton-linearization}
\end{align}
Here \(P\) and \(DP\) are the affine contact operator and its differential
defined in \eqref{eq:contact-operator}--\eqref{eq:contact-operator-differential}.
In particular, the argument of \(DP\) is a direction rather than a base point.
The remainder is
\begin{align}
R_h(v;\delta,w) &= \gamma \Big(\int_0^1 (1-t)\, \varphi''_{+,s} ( P(v + t \delta) )\, dt\, (DP(\delta))^2,DP(w)\Big)_{\Gamma_C}
\label{eq:newton-remainder}
\end{align}
Newton's method then reads: given $u_{h,0}\in V_h$, determine
$u_{h,n} = u_{h,n-1} + \delta_n$, where $\delta_n\in V_h$ is the
solution to
\begin{align}
L_h(u_{h,n-1}; \delta_n,v) = l(v) - A_{h,s}(u_{h,n-1},v) \qquad \forall v \in V_h
\label{eq:newton-step}
\end{align}

To derive the form of the remainder, we use the fundamental theorem of
calculus, integration by parts with $t-1$ as antiderivative of $1$, and the
affinity of $P$, so that $P(v+t\delta) = P(v) + t\, DP(\delta)$, to obtain the
Taylor formula
\begin{align}
\varphi_{+,s}(P(v+\delta))
&=\varphi_{+,s}(P(v)) + \varphi'_{+,s}(P(v))\,DP(\delta)
\nonumber \\
&\qquad
- \int_0^1 (t-1) \varphi''_{+,s}(P(v + t \delta))(DP(\delta))^2\, dt
\label{eq:regularized-positive-part-taylor-formula}
\end{align}

For the analysis of the remainder we record the elementary identities
\begin{align}
\varphi'_{+,s}(w) = 1/2 + \frac{w/4}{\sqrt{w^2/4 + s}}
\label{eq:regularized-positive-part-first-derivative}
\end{align}
\begin{align}
\varphi''_{+,s}(w) =\frac{1/4}{\sqrt{w^2/4 + s}}\frac{s}{w^2/4 + s}
\label{eq:regularized-positive-part-second-derivative}
\end{align}
Thus \(\varphi'_{+,s}\) is strictly increasing and satisfies
\begin{equation}
\lim_{w \rightarrow -\infty} \varphi'_{+,s}(w) = 0,
\qquad
\lim_{w \rightarrow \infty} \varphi'_{+,s}(w) = 1
\label{eq:regularized-positive-part-derivative-limits}
\end{equation}
Furthermore,
\begin{align}
\varphi_{+,s}(w) - \varphi_{+,0}(w) &= \frac{s}{\sqrt{w^2/4 + s} + \sqrt{w^2/4} } \leq \sqrt{s}
\label{eq:regularized-positive-part-uniform-error}
\end{align}

\paragraph{Stable Evaluation.}
For floating-point evaluation, the defining expression loses relative
accuracy when \(w<0\) and \(s\ll w^2\). A cancellation-free representation is
obtained by setting
\begin{equation}
q_s(w)=\operatorname{hypot}(|w|,2\sqrt{s})=\sqrt{w^2+4s}
\label{eq:stable-smoothing-root}
\end{equation}
and evaluating
\begin{equation}
\varphi_{+,s}(w)=
\begin{cases}
(w+q_s(w))/2, & w\geq0\\
2s/(q_s(w)-w), & w<0
\end{cases}
\label{eq:stable-regularized-positive-part}
\end{equation}
The derivatives can then be evaluated without subtractive cancellation as
\begin{align}
\varphi'_{+,s}(w)&=\frac{\varphi_{+,s}(w)}{q_s(w)}
\label{eq:stable-regularized-positive-part-first-derivative}\\
\varphi''_{+,s}(w)&=\frac{2s}{q_s(w)^3}
\label{eq:stable-regularized-positive-part-second-derivative}
\end{align}

\begin{lem}[Remainder estimate]
\label{lem:newton-remainder-estimate}
It holds
\begin{align}
\lvert R_h(v;\delta, w)\rvert
\lesssim \gamma s^{-1/2} \|DP(\delta)\|^2_{L^4(\Gamma_C)} \|DP(w)\|_{\Gamma_C}
\label{eq:newton-remainder-estimate}
\end{align}
\end{lem}
\begin{proof} We have the estimates
\begin{align}
\lvert R_h(v;\delta,w)\rvert
&= \gamma \Big\lvert\Big(\int_0^1  (1-t) \varphi''_{+,s} ( P(v + t \delta) )\, dt\, (DP(\delta))^2,DP(w)\Big)_{\Gamma_C}\Big\rvert
\nonumber \\
&\leq \gamma \Big\|\int_0^1  (1-t) \varphi''_{+,s} ( P(v + t \delta) )\, dt\Big\|_{L^\infty(\Gamma_C)} \|(DP(\delta))^2\|_{L^2(\Gamma_C)}
\|DP(w)\|_{L^2(\Gamma_C)}
\nonumber \\
&\lesssim 
\gamma s^{-1/2} \|DP(\delta)\|^2_{L^4(\Gamma_C)}\|DP(w)\|_{L^2(\Gamma_C)}
\end{align}
Here we used the estimate, at each $x \in \Gamma_C$,
\begin{align}
0 \leq \int_0^1 (1-t) \varphi''_{+,s} ( P(v + t \delta) )\, dt 
&\leq 
\sup_{y \in \IR} \varphi''_{+,s}(y) = \frac{1}{4\sqrt{s}}
\end{align}
which follows since $\varphi''_{+,s}(y)$ attains its maximum at $y=0$.
\end{proof}


\subsection{Coercivity and Solvability}
\label{subsec:newton-coercivity-solvability}

We next prove that the linearization is coercive uniformly in \(s\). This
property yields both unique solvability of the discrete problem and a positive
definite Newton matrix. We use the standard elementwise trace inverse
inequality, valid for \(w\in V_h\),
\begin{align}
\|\sigma_n(w)\|^2_{\Gamma_C} \leq C_I h^{-1} \|\nabla w\|^2_{\Omega}
\end{align}
or equivalently
\begin{align}\label{eq:inverse}
\gamma^{-1} \|\sigma_n(w)\|^2_{\Gamma_C}
\leq \frac{C_I}{\gamma_0} \|\nabla w\|^2_{\Omega}
\end{align}
the second form following from $\gamma = \gamma_0/h$.

\begin{lem}[Coercivity of the linearization]\label{lem:coercive}
It holds
\begin{align}\label{eq:coercive}
L_h(v;\delta,\delta) \geq C \|\nabla \delta\|^2_{\Omega}
\qquad \forall\, v,\delta \in V_h
\end{align}
with the constant
\begin{align}\label{eq:cconst}
C = 1 - \frac{C_I}{\gamma_0}
\end{align}
independent of $v$, of $s$ and of $h$. In particular $C > 0$, and the
linearized form is coercive, provided $\gamma_0 > C_I$.
\end{lem}
\begin{proof}
We have $L_h(v;\delta,\delta) = a_h(\delta,\delta)
+ \gamma( \varphi'_{+,s}(P(v))\, DP(\delta), DP(\delta))_{\Gamma_C}$.
The second term is nonnegative because \(\varphi'_{+,s}\geq0\) and the
remaining factor is a square. For the first term the inverse inequality
\eqref{eq:inverse} gives
\begin{align}
a_h(\delta,\delta) = \|\nabla \delta\|^2_{\Omega}
- \gamma^{-1}\|\sigma_n(\delta)\|^2_{\Gamma_C}
\geq \Big(1 - \frac{C_I}{\gamma_0}\Big) \|\nabla \delta\|^2_{\Omega}
= C \|\nabla \delta\|^2_{\Omega}
\end{align}
Neither bound involves $v$ or $s$.
\end{proof}

Only the nonnegativity of \(\varphi'_{+,s}\) is used for coercivity; strict
positivity is not required because \(a_h\) is already coercive when
\(\gamma_0>C_I\). Thus smoothing supplies classical differentiability, not
the basic coercivity mechanism. The upper bound
\(\varphi'_{+,s}\leq1\) yields the monotonicity estimate in
Section~\ref{sec:error-estimates}.

The lemma is also the variational counterpart of the Jacobian nonsingularity
obtained, for every iterate and every positive parameter, in noninterior
continuation methods for \(P_0\)-matrix complementarity problems \cite{Ka96};
both results rest on the same property of the smoothing function.
Combining the lemma with \eqref{eq:poincare-inequality} we obtain
$L_h(v;\delta,\delta) \geq c_P C \|\delta\|^2_{H^1(\Omega)}$, and since
$A_{h,s}$
is continuously differentiable in its first argument,
\begin{align}\label{eq:mvt}
A_{h,s}(v_1,w) - A_{h,s}(v_2,w)
= \int_0^1 L_h\big(v_2 + t(v_1 - v_2); v_1 - v_2, w\big)\, dt
\end{align}
so that the choice $w = v_1 - v_2$ exhibits $A_{h,s}$ as strongly monotone on
$V_h$, uniformly in $s$.

\begin{prop}[Unique solvability]
Let $\gamma_0 > C_I$, so that $C > 0$. Then
\begin{align}
c_P C \|v_1 - v_2\|^2_{H^1(\Omega)}
\leq A_{h,s}(v_1, v_1 - v_2) - A_{h,s}(v_2, v_1 - v_2)
\qquad \forall\, v_1, v_2 \in V_h
\end{align}
and the discrete problem has a unique solution $u_h \in V_h$ for every
$s > 0$.
\end{prop}
\begin{proof}
The estimate is \eqref{eq:mvt} with $w = v_1 - v_2$, together with the lemma
and the Poincar\'e inequality, and uniqueness follows at once. For existence,
define $T:V_h \rightarrow V_h$ by
$(T(v),w)_{H^1(\Omega)} = A_{h,s}(v,w) - l(w)$ for all $w \in V_h$. Then $T$ is
continuous and strongly monotone on the finite dimensional space $V_h$, hence
surjective, and $T(u_h) = 0$ has a solution.
\end{proof}

\begin{rem}[Consequences for Newton's method]
The method is the stationarity condition of the functional
\begin{align}
E_{h,s}(v) = \frac12 a(v,v)
- \frac{1}{2\gamma}\|\sigma_n(v)\|^2_{\Gamma_C}
- l(v) + \gamma\big(1, \Phi_s(P(v))\big)_{\Gamma_C},
\qquad \Phi_s' = \varphi_{+,s}, \quad \Phi_s(0) = 0
\end{align}
as follows since $P$ is affine with differential $DP$; for $s=0$ we have
$\Phi_0(w) = [w]^2_+/2$, and $E_{h,0}$ is precisely the reduced augmented
Lagrangian \(\mcL_A\) of Section~\ref{sec:finite-element-method}. By the lemma
\(E_{h,s}\) is strongly convex on \(V_h\), uniformly in \(s\).

\emph{Positive definiteness.}
The Newton system is symmetric and positive definite at every iterate and for
every \(s>0\). Each step is therefore well defined and may be computed by a
conjugate gradient method.

\emph{Globalization.}
The Newton direction \(\delta\) satisfies
\(\langle E_{h,s}'(v),\delta\rangle=-L_h(v;\delta,\delta)<0\) and is therefore a
descent direction. A standard backtracking line search satisfying an Armijo
condition gives the usual globalized Newton method, with convergence under the
standard finite-dimensional line-search hypotheses. The computations in
Section~\ref{sec:numerical-examples} use full steps from \(u_{h,0}=0\), so no
globalization is needed there.

\emph{Local convergence.}
For each fixed \(h\) and \(s>0\), the local rate is quadratic.
Lemma~\ref{lem:newton-remainder-estimate} exhibits the factor
\(\gamma s^{-1/2}\), but converting its boundary \(L^4\) norms to the discrete
energy norm introduces mesh-dependent inverse constants. Consequently, the
remainder estimate predicts deterioration of a sufficient local convergence
radius as \(s\to0\), but does not provide an \(h\)-uniform radius. Neither the
solvability of the discrete problem nor the definiteness of the Newton system
deteriorates in the nonsmooth limit.
\end{rem}


\section{Error Estimates}
\label{sec:error-estimates}

We compare the finite element solution \(u_h\) with a continuous central-path
solution \(u_\mu\), and then estimate the distance from \(u_\mu\) to the
Signorini solution \(u\). The key observation is that the regularized method
is exactly consistent with the central-path problem at \(\mu=\gamma s\). The
finite element estimate therefore contains no regularization consistency
term; the effect of the regularization enters only through the continuous
central-path estimate.


\subsection{Central-Path Framework}
\label{subsec:central-path-framework}

The logarithmic-barrier calculation in
Section~\ref{sec:finite-element-method} motivates a continuous barrier
problem on the trace constraint. To formulate it in the natural energy space,
define the extended convex function
\begin{equation}
\Phi(t)\coloneqq
\begin{cases}
-\ln t, & t>0,\\
+\infty, & t\leq0
\end{cases}
\label{eq:extended-logarithmic-barrier}
\end{equation}
and the barrier functional and its effective domain by
\begin{align}
\mathcal B(v)&\coloneqq
\int_{\Gamma_C}\Phi(g-v)\,ds
\label{eq:continuous-barrier-functional}\\
\mathcal D_B&\coloneqq
\{v\in V:\mathcal B(v)<+\infty\}
\label{eq:continuous-barrier-domain}
\end{align}
Because \(g-v\in L^2(\Gamma_C)\), the negative part of
\(\Phi(g-v)\) is integrable. Thus \(\mathcal B(v)\) is well defined in
\(( -\infty,+\infty]\). We impose the barrier-feasibility condition
\begin{equation}
\mathcal D_B\neq\emptyset
\label{eq:barrier-feasibility-condition}
\end{equation}
which is weaker than uniform separation from the obstacle.

\begin{rem}[Barrier feasibility]
\label{rem:barrier-feasibility}
A useful sufficient condition for
\eqref{eq:barrier-feasibility-condition} is the existence of
\(v_0\in K\) and \(\psi\in V\) such that the trace of \(\psi\) is positive
almost everywhere on \(\Gamma_C\) and
\((\ln\psi)_-\in L^1(\Gamma_C)\). Indeed, \(v_0-\psi\in\mathcal D_B\).
For polygonal and polyhedral boundary decompositions, an explicit choice is
\begin{equation}
\psi(x)=\operatorname{dist}(x,\overline\Gamma_D)
\label{eq:barrier-feasibility-distance-function}
\end{equation}
This Lipschitz function belongs to \(V\), its trace is positive in the
relative interior of \(\Gamma_C\), and it vanishes only algebraically at
\(\overline\Gamma_D\cap\overline\Gamma_C\). The boundary singularity
\(-\ln\psi\) is locally integrable there. Consequently,
\(v_0-\psi\in\mathcal D_B\). The condition therefore permits the gap to
approach zero at the Dirichlet--contact interface.
\end{rem}

Set
\begin{align}
E(v)&\coloneqq\frac12a(v,v)-l(v)
\label{eq:unregularized-energy}\\
J_\mu(v)&\coloneqq E(v)+\mu\mathcal B(v)
\label{eq:logarithmic-barrier-energy}
\end{align}

\begin{thm}[Well-posedness of the continuous barrier problem]
\label{thm:continuous-barrier-well-posedness}
Assume \eqref{eq:barrier-feasibility-condition}. For every \(\mu>0\), the
problem
\begin{equation}
u_\mu\in\operatorname*{argmin}_{v\in V}J_\mu(v)
\label{eq:continuous-barrier-minimization}
\end{equation}
has a unique solution. Its gap \(q_\mu=g-u_\mu\) satisfies
\begin{align}
q_\mu&>0
\quad \text{a.e. on }\Gamma_C
\label{eq:central-path-positive-gap}\\
\ln q_\mu&\in L^1(\Gamma_C)
\label{eq:central-path-log-integrability}
\end{align}
and hence \(u_\mu\in K\).
If, in addition, for some \(r>3/2\),
\begin{equation}
u_\mu\in H^r(\Omega)
\quad \text{and}\quad
g\in H^{r-1/2}(\Gamma_C)
\label{eq:fixed-parameter-central-path-regularity}
\end{equation}
then
\begin{equation}
q_\mu^{-1}\in L^2(\Gamma_C)
\label{eq:central-path-reciprocal-integrability}
\end{equation}
and the barrier Euler equation holds in the form
\begin{equation}
a(u_\mu,w)-l(w)+\mu(q_\mu^{-1},w)_{\Gamma_C}=0
\quad \forall w\in V
\label{eq:central-path-weak-problem}
\end{equation}
\end{thm}

\begin{proof}
For every \(\varepsilon>0\), there is a constant
\(C_\varepsilon\) such that
\begin{equation}
-\ln t\geq-\varepsilon t^2-C_\varepsilon
\quad \forall t>0
\label{eq:logarithm-quadratic-lower-bound}
\end{equation}
The trace inequality, Poincar\'e's inequality, and Young's inequality
therefore show, after choosing \(\varepsilon\) sufficiently small depending
on \(\mu\), that
\begin{equation}
J_\mu(v)\geq c\|v\|_{H^1(\Omega)}^2-C
\quad \forall v\in V
\label{eq:barrier-energy-coercivity}
\end{equation}
with \(c>0\). Thus \(J_\mu\) is coercive.

Let \(v_j\rightharpoonup v\) in \(V\). Compactness of the trace from
\(H^1(\Omega)\) into \(L^2(\Gamma_C)\) gives, after passage to a
subsequence, \(g-v_j\to g-v\) strongly in \(L^2(\Gamma_C)\) and almost
everywhere. By \eqref{eq:logarithm-quadratic-lower-bound}, the integrand
\(\Phi(t)+\varepsilon t^2+C_\varepsilon\) is nonnegative. Fatou's lemma,
together with the strong \(L^2(\Gamma_C)\) convergence of the gaps, gives
\begin{equation}
\mathcal B(v)\leq\liminf_{j\to\infty}\mathcal B(v_j)
\label{eq:barrier-functional-lower-semicontinuity}
\end{equation}
Hence \(J_\mu\) is sequentially weakly lower semicontinuous. It is proper by
\eqref{eq:barrier-feasibility-condition}, so the direct method gives a
minimizer. The functional \(E\) is strongly convex by
\eqref{eq:poincare-inequality}, while \(\mathcal B\) is convex. Their sum is
therefore strongly convex, which proves uniqueness. Finally, finiteness of
\(\mathcal B(u_\mu)\) gives \eqref{eq:central-path-positive-gap}. Since
\(q_\mu\in L^2(\Gamma_C)\), its positive logarithmic part is integrable;
finiteness of the barrier gives integrability of the negative part and hence
\eqref{eq:central-path-log-integrability}.

It remains to prove the final assertion. Under
\eqref{eq:fixed-parameter-central-path-regularity}, the traces of \(u_\mu\)
and \(g\) are continuous on \(\Gamma_C\). On each planar contact face \(F\),
the normal trace satisfies
\(\sigma_n(u_\mu)|_F\in H^{r-3/2}(F)\). Thus the broken normal trace belongs
to \(L^2(\Gamma_C)\), without requiring continuity across polyhedral edges.
Hence the continuous representative of \(q_\mu\) is nonnegative, and the
relatively open set
\(G_\mu=\{x\in\Gamma_C:q_\mu(x)>0\}\) has full surface measure.
Variations with zero trace first give \(-\Delta u_\mu=f\) in \(\Omega\).
For \(\eta\in C_c^\infty(G_\mu)\), choose \(w\in V\) whose trace on
\(\Gamma_C\) is \(\eta\). Since \(q_\mu\) is bounded away from zero on the
support of \(\eta\), both signs of \(u_\mu+tw\) belong to
\(\mathcal D_B\) for sufficiently small \(t\). Differentiating
\(J_\mu(u_\mu+tw)\) at \(t=0\) and using Green's formula gives
\begin{equation}
(\sigma_n(u_\mu)+\mu q_\mu^{-1},\eta)_{\Gamma_C}=0
\quad \forall \eta\in C_c^\infty(G_\mu)
\label{eq:localized-central-path-equation}
\end{equation}
Thus
\begin{equation}
\sigma_n(u_\mu)=-\frac{\mu}{g-u_\mu}<0
\quad \text{a.e. on }\Gamma_C
\label{eq:central-path-flux}
\end{equation}
because \(G_\mu\) has full measure. Since the left-hand side belongs to
\(L^2(\Gamma_C)\), relation \eqref{eq:central-path-flux} proves
\eqref{eq:central-path-reciprocal-integrability}; Green's formula then extends
\eqref{eq:localized-central-path-equation} to all \(w\in V\), proving
\eqref{eq:central-path-weak-problem}.
\end{proof}

Theorem~\ref{thm:continuous-barrier-well-posedness} establishes the central
path in the natural variational sense. Finite logarithmic energy alone does
not imply reciprocal integrability, but the Sobolev regularity already used
for finite element approximation yields the boundary Euler equation. In
particular, \eqref{eq:central-path-flux} gives
\begin{equation}
(u_\mu-g)\sigma_n(u_\mu)=\mu
\quad \text{a.e. on }\Gamma_C
\label{eq:central-path-complementarity}
\end{equation}
By Remark~\ref{rem:interior-point-relation}, these relations are equivalent to
\begin{equation}
\sigma_n(u_\mu)=-\gamma\varphi_{+,s}(P(u_\mu))
\quad s=\mu/\gamma
\label{eq:central-path-nitsche-relation}
\end{equation}
for any \(\gamma>0\).

\begin{lem}[Central-path estimate]
\label{lem:central-path-estimate}
It holds
\begin{equation}
\|u-u_\mu\|_{H^1(\Omega)}^2
\leq c_P^{-1}\lvert\Gamma_C\rvert\mu
\label{eq:central-path-estimate}
\end{equation}
\end{lem}

\begin{proof}
Let \(q=g-u\geq0\) on \(\Gamma_C\), and, for \(0<t<1\), set
\begin{align}
u_t&\coloneqq(1-t)u_\mu+tu
\label{eq:central-path-comparison-function}\\
q_t&\coloneqq g-u_t=(1-t)q_\mu+tq\geq(1-t)q_\mu
\label{eq:central-path-comparison-gap}
\end{align}
Thus \(u_t\in\mathcal D_B\). Minimality of \(u_\mu\) and
\eqref{eq:central-path-comparison-gap} imply
\begin{equation}
0\leq J_\mu(u_t)-J_\mu(u_\mu)
\leq E(u_t)-E(u_\mu)-\mu\lvert\Gamma_C\rvert\ln(1-t)
\label{eq:central-path-minimality-comparison}
\end{equation}
Dividing by \(t\) and letting \(t\to0\) gives
\begin{equation}
\langle E'(u_\mu),u-u_\mu\rangle
\geq-\mu\lvert\Gamma_C\rvert
\label{eq:central-path-one-sided-energy-bound}
\end{equation}
Since \(u_\mu\in K\), the variational inequality
\eqref{eq:signorini-variational-inequality} gives
\begin{equation}
\langle E'(u),u_\mu-u\rangle\geq0
\label{eq:central-path-admissible-test}
\end{equation}
Using \eqref{eq:poincare-inequality},
\eqref{eq:central-path-one-sided-energy-bound}, and
\eqref{eq:central-path-admissible-test}, we obtain
\begin{align}
c_P\|u_\mu-u\|_{H^1(\Omega)}^2
&\leq a(u_\mu-u,u_\mu-u)
\label{eq:central-path-proof-coercivity}\\
&=\langle E'(u_\mu)-E'(u),u_\mu-u\rangle
\label{eq:central-path-proof-energy-difference}\\
&\leq\langle E'(u_\mu),u_\mu-u\rangle
\label{eq:central-path-proof-variational-inequality}\\
&\leq\mu\lvert\Gamma_C\rvert
\label{eq:central-path-proof-conclusion}
\end{align}
\end{proof}

\begin{prop}[Uniform local \(H^2\) regularity]
\label{prop:uniform-local-central-path-regularity}
Let \(F\) be contained in the relative interior of a planar face of
\(\Gamma_C\), and let \(U_0\Subset U_1\) be boundary neighborhoods such that
\(\overline U_1\cap\partial\Omega\Subset F\). Assume that, on
\(U_1\cap\Gamma_C\), the obstacle is the trace of a function
\(G\in H^2(U_1\cap\Omega)\). Then, for every \(\mu_0>0\),
\begin{equation}
\sup_{0<\mu\leq\mu_0}
\|u_\mu\|_{H^2(U_0\cap\Omega)}\leq C
\label{eq:uniform-local-central-path-h2-bound}
\end{equation}
Here \(C\) depends on the nested patches and their cutoff geometry, the
Poincar\'e and trace constants, \(\mu_0\), a fixed
\(v_\star\in\mathcal D_B\), and the indicated norms of \(f\), \(g\), and
\(G\), but not on \(\mu\). In particular, the logarithmic boundary layer does
not destroy local \(H^2\) regularity on neighborhoods that may contain a
free-boundary point of the limiting Signorini solution.
\end{prop}

\begin{proof}
Set \(\psi(t)=\Phi(-t)\) and introduce its Moreau envelope
\begin{equation}
\psi_\varepsilon(t)
\coloneqq\inf_{\xi\in\mathbb R}
\left(\psi(\xi)+\frac{|t-\xi|^2}{2\varepsilon}\right)
\label{eq:moreau-logarithmic-barrier}
\end{equation}
The function \(\psi_\varepsilon\) is finite and continuously differentiable,
and \(\psi_\varepsilon'\) is Lipschitz and monotone. Let
\(u_{\mu,\varepsilon}\) minimize
\begin{equation}
J_{\mu,\varepsilon}(v)
\coloneqq E(v)+\mu\int_{\Gamma_C}\psi_\varepsilon(v-g)\,ds
\label{eq:regularized-barrier-functional}
\end{equation}
It satisfies
\begin{equation}
a(u_{\mu,\varepsilon},w)-l(w)
+\mu(\psi_\varepsilon'(u_{\mu,\varepsilon}-g),w)_{\Gamma_C}=0
\quad \forall w\in V
\label{eq:regularized-central-path-euler-equation}
\end{equation}
Lemma~\ref{lem:moreau-uniform-h1-bound} proves, with all smallness choices
made only in terms of \(\mu_0\), that the minimizers are uniformly bounded in
\(H^1(\Omega)\). On a flat boundary patch, set
\(z=u_{\mu,\varepsilon}-G\) and
\(\beta_{\mu,\varepsilon}=\mu\psi_\varepsilon'\). For every tangential
difference quotient, the boundary contribution is
\begin{equation}
(\delta_h\beta_{\mu,\varepsilon}(z),
\eta^2\delta_hz)_{\Gamma_C}\geq0
\label{eq:local-barrier-difference-monotonicity}
\end{equation}
by monotonicity. Lemma~\ref{lem:local-monotone-boundary-h2} expands the
discrete integration-by-parts calculation, estimates the cutoff commutators
and the \(L^2\) right-hand side, and obtains the local \(H^2\) bound uniformly
in \(\mu\) and \(\varepsilon\). Finally,
Lemma~\ref{lem:moreau-central-path-limit} verifies Mosco convergence of
\eqref{eq:regularized-barrier-functional}, proves strong \(H^1\) convergence
\(u_{\mu,\varepsilon}\to u_\mu\), and passes the local estimate to the limit.
These three lemmas are proved in
Appendix~\ref{sec:local-regularity-details} and give
\eqref{eq:uniform-local-central-path-h2-bound}.
\end{proof}

\begin{rem}[Local versus global regularity]
\label{rem:local-versus-global-central-path-regularity}
The monotonicity mechanism in
\eqref{eq:local-barrier-difference-monotonicity} is analogous to that used for
uniformly regular penalized boundary-obstacle approximations
\cite{MeNi05,Ja17}. Proposition~\ref{prop:uniform-local-central-path-regularity}
is local and deliberately avoids
\(\overline\Gamma_D\cap\overline\Gamma_C\) and the edges of polyhedral contact
faces. At those sets, mixed-boundary corner and edge singularities determine
the available global Sobolev exponent; in two-dimensional polygons, their
dependence on the opening angle is described in \cite{ApNi20}. Thus the
proposition isolates the logarithmic layer from the geometric obstruction but
does not replace the global assumption
\eqref{eq:uniform-central-path-regularity}. Obtaining uniform estimates above
\(H^2\), up to the generic \(H^{5/2-\epsilon}\) ceiling, would require a
separate parameter-uniform analysis near the free boundary of the limiting
Signorini solution.
\end{rem}


\subsection{Quasi-optimal Approximation}
\label{subsec:quasi-optimal-approximation}

For the finite element analysis we let $e = u_\mu - u_h$ and, for arbitrary
$v \in V_h$, write $e = \eta + e_h$ where $\eta = u_\mu - v$ and
$e_h = v - u_h \in V_h$. We assume in this subsection that
\eqref{eq:fixed-parameter-central-path-regularity} holds. Thus
Theorem~\ref{thm:continuous-barrier-well-posedness} supplies the boundary
Euler equation, \(\sigma_n(u_\mu)\in L^2(\Gamma_C)\), and Green's formula.
Uniform
Sobolev regularity as \(\mu\to0\) is imposed only in
Corollary~\ref{cor:regularity-dependent-error-estimate}. We use the notation
\begin{equation}
\Delta \varphi_s = \varphi_{+,s}(P(u_\mu)) - \varphi_{+,s}(P(u_h))
\label{eq:regularized-contact-difference}
\end{equation}

The analysis rests on three observations. First, \emph{consistency}: by
Green's formula $a(u_\mu,w) = l(w) + (\sigma_n(u_\mu), w)_{\Gamma_C}$
and the central path relation
$\sigma_n(u_\mu) = -\gamma\varphi_{+,s}(P(u_\mu))$ we obtain, for all
$w \in V_h$,
\begin{align}
a_h(u_\mu,w) + \gamma (\varphi_{+,s}(P(u_\mu)), DP(w))_{\Gamma_C}
&= a(u_\mu,w) - \gamma^{-1}(\sigma_n(u_\mu), \sigma_n(w))_{\Gamma_C}
\nonumber \\
&\qquad
- (\sigma_n(u_\mu), w - \sigma_n(w)/\gamma)_{\Gamma_C}
\nonumber \\
&= a(u_\mu,w) - (\sigma_n(u_\mu), w)_{\Gamma_C}= l(w)
\end{align}
that is, the regularized method is consistent with the central path problem at
$\mu = \gamma s$. Subtracting the method
$a_h(u_h,w) + \gamma(\varphi_{+,s}(P(u_h)), DP(w))_{\Gamma_C} = l(w)$
yields the error equation
\begin{align}
a_h(e,w) + \gamma(\Delta \varphi_s, DP(w))_{\Gamma_C} = 0
\qquad \forall w \in V_h
\end{align}
Secondly, \emph{monotonicity}: since $0 \leq \varphi'_{+,s} \leq 1$, the mean
value theorem gives, for all $x,y \in \IR$,
\begin{align}
(\varphi_{+,s}(x) - \varphi_{+,s}(y))^2
= \varphi'_{+,s}(\xi) (x-y) (\varphi_{+,s}(x) - \varphi_{+,s}(y))
\leq (x-y)(\varphi_{+,s}(x) - \varphi_{+,s}(y))
\end{align}
where we used that $(x-y)(\varphi_{+,s}(x) - \varphi_{+,s}(y)) =
\varphi'_{+,s}(\xi)(x-y)^2 \geq 0$. Since $P$ is affine we have
$P(u_\mu) - P(u_h) = DP(e)$, and applying the inequality pointwise on
$\Gamma_C$ we obtain the monotonicity estimate
\begin{align}
\gamma \|\Delta\varphi_s\|^2_{\Gamma_C}
\leq \gamma (\Delta \varphi_s, DP(e))_{\Gamma_C}
\end{align}
Thirdly, the \emph{inverse inequality} \eqref{eq:inverse}, which
for $w \in V_h$ and $\gamma = \gamma_0/h$ reads
\begin{align}
\gamma^{-1} \|\sigma_n(w)\|^2_{\Gamma_C}
\leq \frac{C_I}{\gamma_0} \|\nabla w\|^2_{\Omega}
\end{align}

\begin{thm}[Quasi-optimal approximation of the central path]
\label{thm:central-path-quasi-optimality}
For $\gamma_0$ sufficiently large it holds
\begin{align}
\|u_\mu - u_h\|^2_{H^1(\Omega)} + \gamma \|\Delta \varphi_s \|^2_{\Gamma_C}
\lesssim
\inf_{v \in V_h} \Big( \|u_\mu - v\|^2_{H^1(\Omega)}
+ \gamma \|u_\mu - v\|^2_{\Gamma_C}
+ \gamma^{-1} \|\sigma_n(u_\mu - v)\|^2_{\Gamma_C} \Big)
\end{align}
In particular, the regularized method is quasi-optimal with respect to the
central path solution, with constants independent of $s$.
\end{thm}
\begin{proof}
Using \eqref{eq:poincare-inequality} and splitting the error,
\begin{align}
\|e\|^2_{H^1(\Omega)} \lesssim a(e,e) = a(e,\eta) + a(e,e_h)
\end{align}
For the second term, using $a(e,e_h) = a_h(e,e_h) +
\gamma^{-1}(\sigma_n(e),\sigma_n(e_h))_{\Gamma_C}$, the error equation
with $w = e_h$, the splitting $DP(e_h) = DP(e) - DP(\eta)$, and the
monotonicity estimate,
\begin{align}
a(e,e_h) &= \gamma^{-1}(\sigma_n(e),\sigma_n(e_h))_{\Gamma_C}
- \gamma(\Delta\varphi_s, DP(e))_{\Gamma_C}
+ \gamma(\Delta\varphi_s, DP(\eta))_{\Gamma_C}
\nonumber \\
&\leq \gamma^{-1}(\sigma_n(e),\sigma_n(e_h))_{\Gamma_C}
- \gamma\|\Delta\varphi_s\|^2_{\Gamma_C}
+ \gamma(\Delta\varphi_s, DP(\eta))_{\Gamma_C}
\end{align}
and thus
\begin{align}
\|e\|^2_{H^1(\Omega)} + \gamma\|\Delta\varphi_s\|^2_{\Gamma_C}
\lesssim
\underbrace{a(e,\eta)}_{I}
+ \underbrace{\gamma^{-1}(\sigma_n(e),\sigma_n(e_h))_{\Gamma_C}}_{II}
+ \underbrace{\gamma(\Delta\varphi_s, DP(\eta))_{\Gamma_C}}_{III}
\end{align}
\emph{Term \(I\).}
By Cauchy--Schwarz and Young's inequality,
\begin{align}
I \leq \|e\|_{H^1(\Omega)} \|\eta\|_{H^1(\Omega)}
\leq \delta \|e\|^2_{H^1(\Omega)} + \frac{1}{4\delta}\|\eta\|^2_{H^1(\Omega)}
\end{align}
\emph{Term \(II\).}
Splitting $\sigma_n(e) = \sigma_n(\eta) + \sigma_n(e_h)$ and
using the inverse inequality together with $\|\nabla e_h\|^2_\Omega \leq
2\|\nabla e\|^2_\Omega + 2\|\nabla \eta\|^2_\Omega$,
\begin{align}
II &= \gamma^{-1}(\sigma_n(\eta),\sigma_n(e_h))_{\Gamma_C}
+ \gamma^{-1}\|\sigma_n(e_h)\|^2_{\Gamma_C}
\nonumber \\
&\leq \frac12 \gamma^{-1}\|\sigma_n(\eta)\|^2_{\Gamma_C}
+ \frac32 \gamma^{-1}\|\sigma_n(e_h)\|^2_{\Gamma_C}
\nonumber \\
&\leq \frac12 \gamma^{-1}\|\sigma_n(\eta)\|^2_{\Gamma_C}
+ \frac{3 C_I}{\gamma_0}\Big(\|\nabla e\|^2_{\Omega} + \|\nabla \eta\|^2_{\Omega}\Big)
\end{align}
\emph{Term \(III\).}
By Young's inequality and $\|DP(\eta)\|^2_{\Gamma_C}
\leq 2\|\eta\|^2_{\Gamma_C} +
2\gamma^{-2}\|\sigma_n(\eta)\|^2_{\Gamma_C}$,
\begin{align}
III \leq \frac{\gamma}{4}\|\Delta\varphi_s\|^2_{\Gamma_C}
+ \gamma\|DP(\eta)\|^2_{\Gamma_C}
\leq \frac{\gamma}{4}\|\Delta\varphi_s\|^2_{\Gamma_C}
+ 2\gamma \|\eta\|^2_{\Gamma_C}
+ 2\gamma^{-1}\|\sigma_n(\eta)\|^2_{\Gamma_C}
\end{align}
Collecting the estimates and taking $\delta$ sufficiently small and $\gamma_0$
sufficiently large, the terms $\delta \|e\|^2_{H^1(\Omega)}$,
$\gamma_0^{-1}\|\nabla e\|^2_\Omega$, and
$\tfrac{\gamma}{4}\|\Delta\varphi_s\|^2_{\Gamma_C}$ may be absorbed
into the left-hand side, and the theorem follows by taking the infimum over
$v \in V_h$.
\end{proof}

\begin{lem}[Fractional boundary-flux approximation]
\label{lem:fractional-boundary-flux-approximation}
Let \(3/2<r\leq k+1\), and let \(\pi_h\) satisfy the local approximation
property \eqref{eq:local-interpolation-estimate}. Then
\begin{equation}
\|\sigma_n(v-\pi_hv)\|_{\Gamma_C}
\lesssim h^{r-3/2}\|v\|_{H^r(\Omega)}
\quad \forall v\in H^r(\Omega) \label{eq:fractional-boundary-flux-approximation}
\end{equation}
\end{lem}

\begin{proof}
Let \(F=\partial T\cap\Gamma_C\) be a contact face and set
\(\eta=v-\pi_hv\). For \(1/2<t\leq1\), the scaled fractional trace
inequality \cite[Lemma~7.2]{ErGu17} gives
\begin{equation}
\|w\|_F
\lesssim h^{-1/2}\|w\|_T
+h^{t-1/2}\|w\|_{H^t(T)}
\quad \forall w\in H^t(T) \label{eq:scaled-fractional-trace-inequality}
\end{equation}
When \(3/2<r\leq2\), take \(w=\nabla\eta\) and \(t=r-1\), and then use
\eqref{eq:local-interpolation-estimate} with \(m=1\) and \(m=r\), to obtain
\begin{align}
\|\sigma_n(\eta)\|_F
&\lesssim h^{-1/2}\|\nabla\eta\|_T
+h^{r-3/2}\|\nabla\eta\|_{H^{r-1}(T)}
\label{eq:fractional-boundary-flux-local}\\
&\lesssim h^{r-3/2}\|v\|_{H^r(N_h(T))}
\label{eq:fractional-boundary-flux-local-conclusion}
\end{align}
For \(r>2\), the same local conclusion follows from the standard
element trace inequality applied to \(\nabla\eta\), together with
\eqref{eq:local-interpolation-estimate} for \(m=1\) and \(m=2\).
Squaring, summing over the contact faces, and using the uniformly bounded
overlap of the patches \(N_h(T)\) proves
\eqref{eq:fractional-boundary-flux-approximation}.
\end{proof}

\begin{cor}[Regularity-dependent error estimate]
\label{cor:regularity-dependent-error-estimate}
Let \(3/2<r\leq k+1\), assume that
\(g\in H^{r-1/2}(\Gamma_C)\), and suppose that
\begin{equation}
\sup_{0<\mu\leq\mu_0}\|u_\mu\|_{H^r(\Omega)}\leq C_r
\label{eq:uniform-central-path-regularity}
\end{equation}
With \(\gamma=\gamma_0/h\) and constants depending on \(\gamma_0\), it holds
\begin{equation}
\|u_\mu-u_h\|_{H^1(\Omega)}
+\gamma^{1/2}\|\Delta\varphi_s\|_{\Gamma_C}
\lesssim h^{r-1}\|u_\mu\|_{H^r(\Omega)}
\label{eq:central-path-discretization-error}
\end{equation}
Combining this estimate with Lemma~\ref{lem:central-path-estimate} gives
\begin{equation}
\|u-u_h\|_{H^1(\Omega)}
\lesssim h^{r-1}\|u_\mu\|_{H^r(\Omega)}+\mu^{1/2}
\quad \mu=\gamma s=\gamma_0s/h
\label{eq:regularity-dependent-total-error}
\end{equation}
Consequently, the sufficient balance
\begin{equation}
s\lesssim h^{2r-1}
\label{eq:regularity-dependent-smoothing-balance}
\end{equation}
or, equivalently, \(\mu\lesssim h^{2r-2}\), retains the rate
\(h^{r-1}\). In particular, if \(r=k+1\), the sufficient full-regularity
balance
\begin{equation}
s\lesssim h^{2k+1}
\label{eq:full-regularity-smoothing-balance}
\end{equation}
retains the full polynomial-order rate \(h^k\).
\end{cor}

\begin{proof}
Let \(v=\pi_hu_\mu\). The interpolation and elementwise trace estimates give
\begin{align}
\|\eta\|_{H^1(\Omega)}
&\lesssim h^{r-1}\|u_\mu\|_{H^r(\Omega)}
\label{eq:regularity-dependent-interpolation-energy}\\
\|\eta\|_{\Gamma_C}
&\lesssim h^{r-1/2}\|u_\mu\|_{H^r(\Omega)}
\label{eq:regularity-dependent-interpolation-trace}\\
\|\sigma_n(\eta)\|_{\Gamma_C}
&\lesssim h^{r-3/2}\|u_\mu\|_{H^r(\Omega)}
\label{eq:regularity-dependent-interpolation-flux}
\end{align}
Here the first two bounds follow from
\eqref{eq:local-interpolation-estimate} and the standard element trace
inequality, while the third is
Lemma~\ref{lem:fractional-boundary-flux-approximation}. In particular, its
proof remains valid for the fractional range \(3/2<r<2\).
Inserting these bounds in
Theorem~\ref{thm:central-path-quasi-optimality} proves
\eqref{eq:central-path-discretization-error}; the remaining statements follow
from \eqref{eq:central-path-estimate} and \(\mu=\gamma_0s/h\).
\end{proof}


\subsection{Interpretation and Consequences}
\label{subsec:error-estimate-interpretation}

The balance \(\mu\lesssim h^{2r-2}\) states that the barrier parameter is of
the order of the square of the available energy-norm approximation error, in
accordance with its interpretation as a complementarity gap.

\begin{rem}[Regularity of the central path]
The uniform bound \eqref{eq:uniform-central-path-regularity} is an assumption
on the entire central-path family, not merely on each fixed \(u_\mu\). Since
\(u_\mu\to u\) as \(\mu\to0\), a uniform \(H^{k+1}(\Omega)\) bound requires
the limit to have the same regularity. A generic Signorini solution belongs
only to \(H^{5/2-\epsilon}(\Omega)\), so the basic Sobolev argument guarantees
rates below \(3/2\), rather than an exact generic rate of \(3/2\), for
quadratic and higher-order elements. Taking
\(r=5/2-\epsilon\) in
Corollary~\ref{cor:regularity-dependent-error-estimate} gives the balance
\(s\lesssim h^{4-2\epsilon}\). The more restrictive choice
\(s\lesssim h^{2k+1}\) remains sufficient whenever \(r\leq k+1\), but need
not be the least restrictive sufficient choice.

Proposition~\ref{prop:uniform-local-central-path-regularity} proves that the
central path is uniformly \(H^2\) on neighborhoods that may contain
free-boundary points of the limiting Signorini solution in the relative
interior of a planar contact face. Consequently, the reciprocal logarithmic
law itself does not obstruct the regularity needed for the local \(P_1\)
estimate. The global hypothesis in
\eqref{eq:uniform-central-path-regularity} remains necessary because the
Dirichlet--contact interface, corners, and polyhedral edges are excluded from
that proposition.

The \(P_2\) energy-error data for Example B in
Section~\ref{sec:numerical-examples} show a finest-mesh rate near \(3/2\).
Such a rate may follow from additional weighted or Besov regularity, but it is
not asserted by the present Sobolev estimate. Formally balancing
\(\rho(u_\mu)\sigma_n(u_\mu)=\mu\) with the generic behaviours
\(\sigma_n\sim d^{1/2}\) and \(g-u\sim d^{3/2}\), where \(d\) is the distance
from the free boundary of the limiting Signorini solution, suggests a barrier
layer of width
\(O(\mu^{1/2})\). A rigorous analysis of this layer is left for future work.
\end{rem}

\begin{rem}[Dimensional scaling]
The condition
$s \lesssim h^{2k+1}$ is stated for nondimensional quantities. Since $P(v)$
has the dimension of the solution, $s$ has the dimension of the solution
squared, whereas $\gamma = \gamma_0/h$ is large. Let $L$ denote a reference
length for $\Omega$ and $U$ a reference magnitude of $u$. Writing the central
path estimate as $\|u - u_\mu\|_{H^1(\Omega)} \lesssim (L\mu)^{1/2}$ and
requiring this to be dominated by the discretization error
$h^k \|u\|_{H^{k+1}(\Omega)} \sim U (h/L)^k$ gives
\begin{align}
\mu \lesssim \frac{U^2}{L}\Big(\frac{h}{L}\Big)^{2k}
\qquad \text{or equivalently} \qquad
\boxed{s \lesssim \frac{U^2}{\gamma_0}\Big(\frac{h}{L}\Big)^{2k+1}}
\end{align}
which reduces to the stated condition for $L = U = \gamma_0 = 1$. The factor
$\gamma_0^{-1}U^2L^{-(2k+1)}$ is easily overlooked and may be far from unity,
in particular since $\gamma_0$ is chosen large to secure coercivity. In the
computations of Section~\ref{sec:numerical-examples}, \(L\) is of order unity,
\(\gamma_0=10\) for \(P_1\), and \(\gamma_0=20\) for \(P_2\). The choice
$s = h^{2k+1}/4$ lies close to this bound, consistent with the observation
made there that at this balance the regularization contributes at the same
order as the discretization error.
\end{rem}

\begin{cor}[Discrete contact diagnostics]
\label{cor:discrete-contact-diagnostics}
Under the assumptions of
Corollary~\ref{cor:regularity-dependent-error-estimate}, define
\(\lambda_h=-\gamma\varphi_{+,s}(P(u_h))\). Then
\begin{align}
\gamma^{-1/2}\|\sigma_n(u_h)-\lambda_h\|_{\Gamma_C}
&\lesssim h^{r-1}\|u_\mu\|_{H^r(\Omega)}
\label{eq:stress-multiplier-discrepancy}\\
\|(u_h-g)_+\|_{\Gamma_C}
&\lesssim h^{r-1/2}\|u_\mu\|_{H^r(\Omega)}
\label{eq:discrete-penetration-rate}\\
\|(g-u_h)(-\lambda_h)-\mu\|_{L^1(\Gamma_C)}
&\lesssim h^{r-1/2}\|u_\mu\|_{H^r(\Omega)}^2
\label{eq:discrete-perturbed-complementarity-residual}
\end{align}
Consequently,
\begin{equation}
\|(g-u_h)(-\lambda_h)\|_{L^1(\Gamma_C)}
\lesssim
h^{r-1/2}\|u_\mu\|_{H^r(\Omega)}^2
+\mu\lvert\Gamma_C\rvert
\label{eq:discrete-complementarity-residual}
\end{equation}
\end{cor}

\begin{proof}
Let \(v=\pi_hu_\mu\) and use the splitting
\(e=u_\mu-u_h=\eta+e_h\) from the proof of
Theorem~\ref{thm:central-path-quasi-optimality}. The inverse inequality,
\eqref{eq:regularity-dependent-interpolation-flux}, and
\eqref{eq:central-path-discretization-error} give
\begin{equation}
\gamma^{-1/2}\|\sigma_n(e)\|_{\Gamma_C}
\lesssim h^{r-1}\|u_\mu\|_{H^r(\Omega)}
\label{eq:total-flux-error-bound}
\end{equation}
Combining this with
\(\sigma_n(u_\mu)-\lambda_h=-\gamma\Delta\varphi_s\) proves
\eqref{eq:stress-multiplier-discrepancy}. The pointwise identities
\begin{align}
\widehat q_h
&\coloneqq(g-u_h)+\gamma^{-1}(\sigma_n(u_h)-\lambda_h)
=\varphi_{+,s}(-P(u_h))>0
\label{eq:corrected-discrete-gap}\\
\widehat q_h(-\lambda_h)&=\mu
\label{eq:corrected-discrete-complementarity}
\end{align}
follow from
\(\varphi_{+,s}(w)\varphi_{+,s}(-w)=s\). The first identity and
the positivity of \(\widehat q_h\) imply
\begin{equation}
\|(u_h-g)_+\|_{\Gamma_C}
\leq\gamma^{-1}\|\sigma_n(u_h)-\lambda_h\|_{\Gamma_C}
\label{eq:discrete-feasibility-l2-bound}
\end{equation}
Together with \eqref{eq:stress-multiplier-discrepancy}, this proves
\eqref{eq:discrete-penetration-rate}. Moreover,
\begin{equation}
(g-u_h)(-\lambda_h)-\mu
=-\gamma^{-1}(\sigma_n(u_h)-\lambda_h)(-\lambda_h)
\label{eq:physical-gap-complementarity-identity}
\end{equation}
By \eqref{eq:central-path-flux}, the trace inequality, and
\eqref{eq:central-path-discretization-error},
\begin{equation}
\gamma^{-1/2}\|\lambda_h\|_{\Gamma_C}
\lesssim h^{1/2}\|u_\mu\|_{H^r(\Omega)}
\label{eq:scaled-discrete-pressure-bound}
\end{equation}
Applying Cauchy--Schwarz to
\eqref{eq:physical-gap-complementarity-identity} proves
\eqref{eq:discrete-perturbed-complementarity-residual}, and the triangle
inequality gives \eqref{eq:discrete-complementarity-residual}.
\end{proof}

The corrected gap \(\widehat q_h\), rather than the physical gap
\(g-u_h\), lies exactly on the discrete central path. The corollary quantifies
both the resulting penetration and the physical complementarity residual.

\begin{rem}[Comparison with a direct estimate]
\label{rem:direct-error-estimate}
Comparing \(u_h\) directly with \(u\), rather than with the central path, leads
to the error equation
\begin{align}
a_h(u - u_h, w)
+ \gamma\big(\varphi_{+,s}(P(u)) - \varphi_{+,s}(P(u_h)),
DP(w)\big)_{\Gamma_C}
= \gamma(\rho_s, DP(w))_{\Gamma_C}
\end{align}
with the consistency error
$\rho_s = \varphi_{+,s}(P(u)) - \varphi_{+,0}(P(u))$.
Estimating this term by Cauchy--Schwarz, using the global bound
$0 \leq \rho_s \leq s^{1/2}$, yields the contribution $\gamma^2 s$ and hence
the stronger condition $s \lesssim h^{2(k+1)}$. The gain of one power of \(h\)
in the corollary stems from exact consistency with the central path: no
consistency error enters the finite
element estimate, and the regularization is accounted for once, at the
continuous level, by Lemma~\ref{lem:central-path-estimate}. The resulting
\(\mu^{1/2}\) bound is sufficient and is sharp for the present argument. The
examples in Section~\ref{sec:numerical-examples} often converge faster, but a
further general relaxation of the condition on \(s\) does not follow from the
global bounds used here.
\end{rem}

\begin{rem}[Contact pressure]
\label{rem:contact-pressure}
Define the discrete contact pressure
\(\lambda_h=-\gamma\varphi_{+,s}(P(u_h))\). Since
\(\sigma_n(u_\mu)=-\gamma\varphi_{+,s}(P(u_\mu))\), the corollary gives
\begin{align}
\|\sigma_n(u_\mu) - \lambda_h\|_{\Gamma_C}
= \gamma \|\Delta\varphi_s\|_{\Gamma_C}
\lesssim \gamma^{1/2} h^{r-1} \|u_\mu\|_{H^r(\Omega)}
= \gamma_0^{1/2} h^{r-3/2} \|u_\mu\|_{H^r(\Omega)}
\end{align}
that is, the usual loss of half an order for the boundary flux, now measured
relative to the central-path pressure. Under full regularity \(r=k+1\), this
is the rate \(h^{k-1/2}\). This estimate is relative to
\(\sigma_n(u_\mu)\), not to the exact contact pressure \(\sigma_n(u)\).
Controlling the distance between those two pressures requires additional
assumptions on the contact zone and is not pursued here. In the smooth
computations of Section~\ref{sec:numerical-examples}, the observed error
relative to \(\sigma_n(u)\) decays like \(h^k\), so the half-order loss is not
seen there. For the singular example, the measured rates decrease under
refinement and remain preasymptotic; the present analysis does not assign an
asymptotic rate to the error relative to \(\sigma_n(u)\).
\end{rem}


\section{Numerical Examples}
\label{sec:numerical-examples}

The experiments examine the regularization error, Newton convergence,
approximation of smooth and singular exact solutions, contact diagnostics, and
robustness with respect to the Nitsche parameter and mesh structure.


\subsection{Baseline Problem and Regularization Error}
\label{subsec:baseline-regularization-error}

We consider the model problem on the unit square $\Omega = (0,1)^2$
with load $f = -1$ for $y < 1/2$ and $f = 1$ for $y \geq 1/2$, homogeneous
Dirichlet conditions on the boundary part $x = 0$, and contact conditions
with $g = 0$ on the remainder of the boundary. The domain is partitioned into
uniform triangular meshes with $h = 1/n$, $n = 8,\dots,256$ for $P_1$ and
$n = 8,\dots,128$ for $P_2$, and we use $\gamma_0 = 10$ for $P_1$ elements
and $\gamma_0 = 20$ for $P_2$ elements.
The nonlinear systems are solved by Newton's method with full steps, starting
from $u_{h,0} = 0$, until the Euclidean residual norm $r_n$ on the free
degrees of freedom satisfies
$r_n < \max(10^{-10},10^{-12}r_0)$.
For $P_1$ we use one-point volume and two-point edge quadrature; for $P_2$
we use a degree-five volume rule and four-point Gauss edge quadrature. The
smoothed positive part is evaluated by the cancellation-free formulas in
Section~\ref{sec:newton-method}. The $P_2$ solution is shown in Figure
\ref{fig:solution-p2}; the constraint is active along the upper part of the
boundary, where the load pushes the solution against the obstacle $g = 0$.

\begin{figure}[H]
\centering
\includegraphics[width=0.70\textwidth]{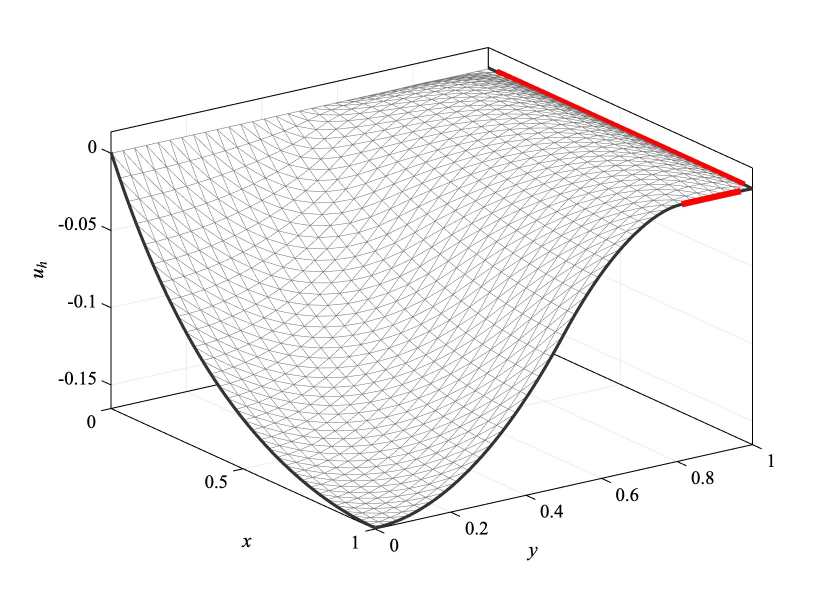}
\caption{Elevation of the $P_2$ solution with the mesh shown ($n = 32$,
$s = h^5/4$). The contact part of the boundary is drawn on top of the
surface: red where $P(u_h)=(u_h-g)-\sigma_n(u_h)/\gamma>0$ at the edge
midpoint, and dark gray where it is inactive. The Dirichlet boundary
$x=0$ is unmarked.}
\label{fig:solution-p2}
\end{figure}
\FloatBarrier

To isolate the effect of the regularization we compute, on each mesh,
both the regularized solution $u_{h,s}$ with $s = h^{\alpha}/4$ for a range
of exponents $\alpha$, and the unregularized solution $u_{h,0}$, obtained by
a primal-dual active set (semismooth Newton) method, and measure
$\lvert u_{h,s} - u_{h,0}\rvert_{H^1(\Omega)}$. Since the difference vanishes
on the Dirichlet boundary, this seminorm is equivalent to the full
$H^1(\Omega)$ norm. Both solutions are computed on the same mesh, so the
common discretization effects are reduced. The resulting difference is a
discrete diagnostic motivated by the continuous central-path distance. The
estimate in Section~\ref{sec:error-estimates} therefore motivates the reference
scale
\begin{align}
\mathcal E_{\mathrm{ref}}(h,s)
&\coloneqq (\gamma s)^{1/2}=O(h^{(\alpha-1)/2})
\label{eq:discrete-regularization-reference-scaling}
\end{align}
The results are shown in Figure \ref{fig:sweep-p1} for $P_1$ and in
Figure \ref{fig:sweep-p2} for $P_2$, together with reference lines of the
motivating slope $(\alpha-1)/2$. In every case the observed discrete
difference decays at least as fast as this reference scaling on the tested
meshes, with measured rates about half an order higher. In particular, at the
sufficient full-regularity balance \(\alpha=2k+1\) in
\eqref{eq:full-regularity-smoothing-balance}, the
observed rates on the finest meshes are $1.45$ for $P_1$ ($\alpha = 3$) and
$3.06$ for $P_2$ ($\alpha = 5$), above the target finite element energy order
$k$. Thus the regularization difference decays faster than the discretization
error in these experiments. In contrast, the choice $\alpha = 2$ for $P_1$
yields the observed rate $0.68 < 1$, so it would dominate an order-one
discretization error.

\begin{figure}[!htbp]
\centering
\begin{subfigure}[t]{0.49\textwidth}
\centering
\includegraphics[width=\textwidth]{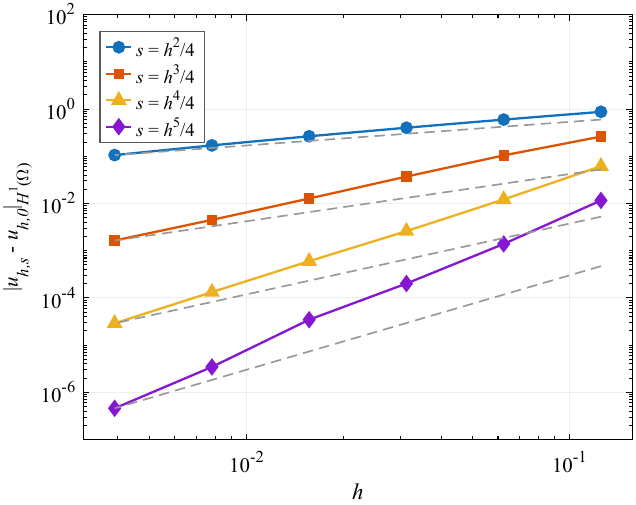}
\caption{$P_1$ elements, $\alpha=2,\dots,5$.}
\label{fig:sweep-p1}
\end{subfigure}
\hfill
\begin{subfigure}[t]{0.49\textwidth}
\centering
\includegraphics[width=\textwidth]{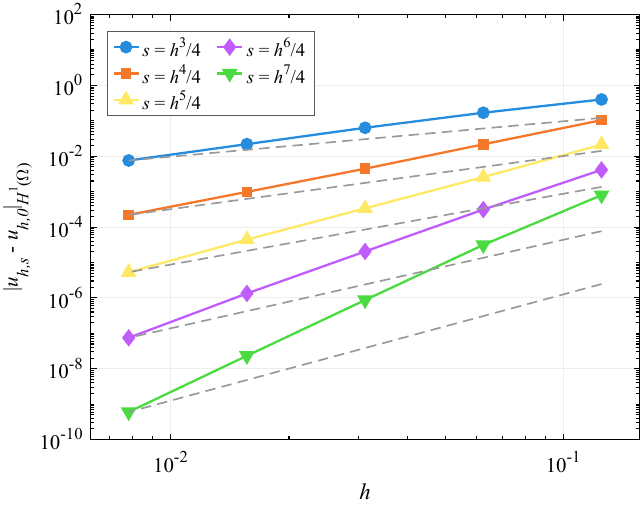}
\caption{$P_2$ elements, $\alpha=3,\dots,7$.}
\label{fig:sweep-p2}
\end{subfigure}
\caption{Regularization error
$\lvert u_{h,s}-u_{h,0}\rvert_{H^1(\Omega)}$ for $s=h^\alpha/4$
(solid, markers), together with reference lines of slope $(\alpha-1)/2$
(dashed), motivated by the continuous central-path estimate. The sufficient
full-regularity balances are $\alpha=3$ for $P_1$ and $\alpha=5$ for $P_2$.}
\label{fig:sweeps}
\end{figure}
\FloatBarrier


\subsection{Exact-Solution Test Problems}
\label{subsec:exact-solution-test-problems}

To assess the convergence rates predicted by
Corollary~\ref{cor:regularity-dependent-error-estimate}, we next consider two problems
with known exact solution on $\Omega = (-1,1)\times(0,1)$, with contact
conditions and $g = 0$ on the bottom edge
$\Gamma_C = (-1,1)\times\{0\}$, and Dirichlet conditions on the remaining
three sides with data taken from the exact solution. These computations use
degree-five volume and four-point edge quadrature for both element orders.
The initial iterate has the prescribed Dirichlet values and zero free degrees
of freedom, and the stopping criterion is
$r_n < \max(10^{-11},10^{-12}r_0)$.

\paragraph{Example A: Smooth Manufactured Solution.}
\label{par:numerical-example-a}
The first test uses the manufactured solution
\begin{align}
u_A = -x_+^3 + y\,(-x)_+^3, \qquad
f = -\Delta u_A =
\begin{cases}
6x, & x > 0
\\
6xy, & x < 0
\end{cases}
\label{eq:exact-solution-a}
\end{align}
On $\Gamma_C$ we have $u_A = -x_+^3 \leq 0$,
$\sigma_n(u_A) = -(-x)_+^3 \leq 0$ and $u_A \sigma_n(u_A) = 0$, so the
Kuhn--Tucker conditions hold exactly; since the problem is strictly convex,
$u_A$ is the solution. The contact set is $[-1,0]$ and the free boundary is
the point $x = 0$. Since $u_A \in C^{2,1}(\Omega) \subset H^3(\Omega)$ the
optimal orders are attainable for both $P_1$ and $P_2$, at the price of a
contact pressure that degenerates cubically at the free boundary.

\paragraph{Example B: Singular Signorini Profile.}
\label{par:numerical-example-b}
The second test uses the classical solution
\begin{align}
u_B = -r^{3/2} \cos(3\theta/2), \qquad f = 0
\label{eq:exact-solution-b}
\end{align}
in polar coordinates centred at the origin, $\theta \in [0,\pi]$. Here
$u_B = 0$ and $\sigma_n(u_B) = -\tfrac32 r^{1/2} < 0$ for $x<0$, while
$u_B = -|x|^{3/2} < 0$ and $\sigma_n(u_B) = 0$ for $x>0$, so the contact set
is again $[-1,0]$, but the contact pressure now exhibits the square root
behaviour generic for the Signorini problem. Consequently
$u_B \in H^{5/2-\epsilon}(\Omega)$ only, consistent with the optimal
$C^{1,1/2}$ regularity of the Signorini problem \cite{AtCa06}, and the
convergence rate is limited
to $\min(k,3/2)$. The computed solution is shown in Figure
\ref{fig:solution-B}; the active set determined by the method coincides with
the exact contact set $[-1,0]$, that is, the free boundary is located at
$x = 0$ to within one element.

\begin{figure}[!htbp]
\centering
\includegraphics[width=0.74\textwidth]{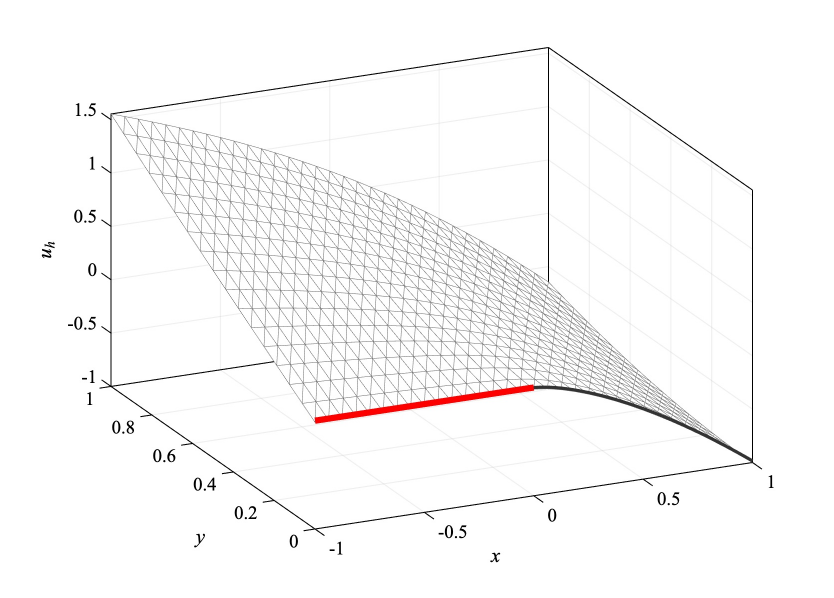}
\caption{Elevation of the $P_2$ solution for Example B ($n=16$,
$s=h^5/4$). Red marks the active contact boundary and dark gray the inactive
part. The solution is flat at the obstacle on $[-1,0]$ and detaches like
$-|x|^{3/2}$ beyond the free boundary at $x=0$.}
\label{fig:solution-B}
\end{figure}
\FloatBarrier


\subsection{Newton Performance and Error Convergence}
\label{subsec:newton-performance-error-convergence}

\paragraph{Newton Performance.}
Figure~\ref{fig:newton} summarizes the nonlinear iteration. In the
regularization study of Subsection~\ref{subsec:baseline-regularization-error},
the iteration count stays between \(7\) and \(15\) over six mesh refinements.
It is essentially independent of the mesh size and increases only mildly as
\(s\) approaches the nonsmooth limit. For comparison, the primal-dual active
set method for the unregularized problem \((s=0)\) required \(5\)--\(12\)
iterations on the same meshes. Here and below, an iteration means one
linearized update; evaluation of the initial residual is not counted.

The lower panels of Figure~\ref{fig:newton} show residual histories for
Example A. The observed order, computed from the Euclidean residual norms
\(r_n\) as
\begin{equation}
q_n = \frac{\log(r_{n+1}/r_n)}{\log(r_n/r_{n-1})}
\label{eq:observed-newton-order}
\end{equation}
is approximately \(1.9\)--\(2.0\) in the final resolved local steps of the
two full-balance runs, before the residual reaches round-off. This behavior is
consistent with quadratic local convergence. Undamped steps were used
throughout; the globalization discussed in Section~\ref{sec:newton-method}
was not needed.

\begin{figure}[!htbp]
\centering
\includegraphics[width=0.96\textwidth]{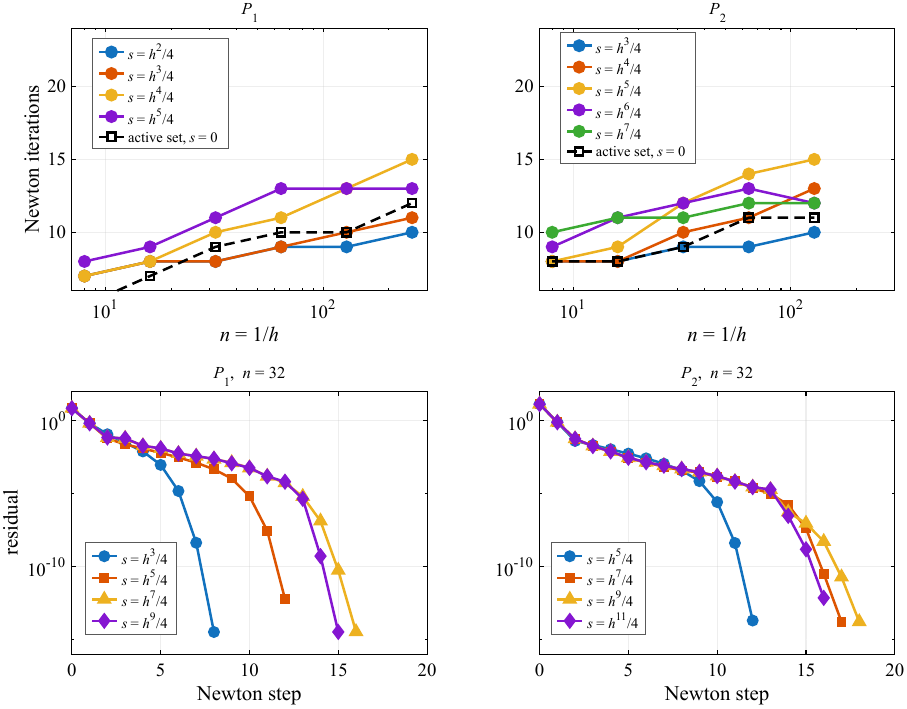}
\caption{Newton performance. Upper panels: iteration counts for the
regularized method with $s=h^\alpha/4$ and for the primal-dual active set
method at $s=0$. Counts refer to linearized updates and exclude the initial
residual evaluation. Lower panels: residual histories for Example A on the
mesh $n=32$, with full Newton steps. For the full-balance runs, the observed
local order is approximately $1.9$--$2.0$ before round-off dominates.}
\label{fig:newton}
\end{figure}
\FloatBarrier

\paragraph{Convergence against Exact Solutions.}
The errors for the sufficient full-regularity choice $s = h^{2k+1}/4$ are
reported in
Figure \ref{fig:exact-conv}.
In the energy norm the rates observed on the finest meshes are $1.07$ and
$1.04$ for $P_1$ in Examples A and B, approaching the expected order
$k = 1$ from above, and $2.29$ for $P_2$ in Example A, slightly above the
expected order $k = 2$. For $P_2$ in Example B the rate decreases from $2.43$
to $1.50$ as the mesh resolves the singularity. This endpoint rate is
consistent with the known approximation behaviour of this particular corner
singularity; the basic Sobolev estimate in Section~\ref{sec:error-estimates}
guarantees rates below \(3/2\). The number of Newton
iterations was between $7$ and $16$ throughout, growing by one or two per
refinement as $s$ decreases with the mesh.

\begin{figure}[!htbp]
\centering
\includegraphics[width=0.96\textwidth]{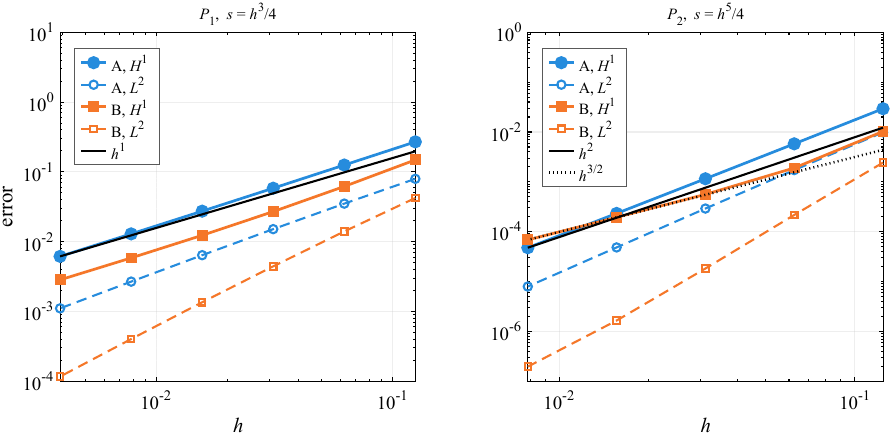}
\caption{Convergence against the exact solutions of Examples A and B in the
energy norm (solid) and $L^2(\Omega)$ norm (dashed), with
$s=h^{2k+1}/4$. For $P_2$, the energy error of Example B has a finest-mesh
rate near \(3/2\), consistent with the regularity of \(u_B\), rather than the
full order \(2\).}
\label{fig:exact-conv}
\end{figure}
\FloatBarrier

\paragraph{Effect of the Smoothing Parameter.}
Figure~\ref{fig:svar} clarifies the role of the full-regularity balance for
\(P_1\) elements in Example A. At
$s = h^{2k+1}/4$ the regularization contributes at the same order as the
discretization error, and indeed the energy norm error on the finest mesh
exceeds the value obtained for negligible regularization by about forty per
cent, but the rate is unaffected. Reducing $s$ further leaves the energy-norm
error essentially unchanged, indicating that the discretization error then
dominates. It does, however, improve the $L^2(\Omega)$ error, which is polluted
by regularization at $s = h^{2k+1}$; the rate $2$ is recovered for
$s \sim h^{2k+3}$. The analysis of Section~\ref{sec:error-estimates} concerns
the energy norm only and makes no claim about the $L^2(\Omega)$ error.

\begin{figure}[!htbp]
\centering
\includegraphics[width=0.96\textwidth]{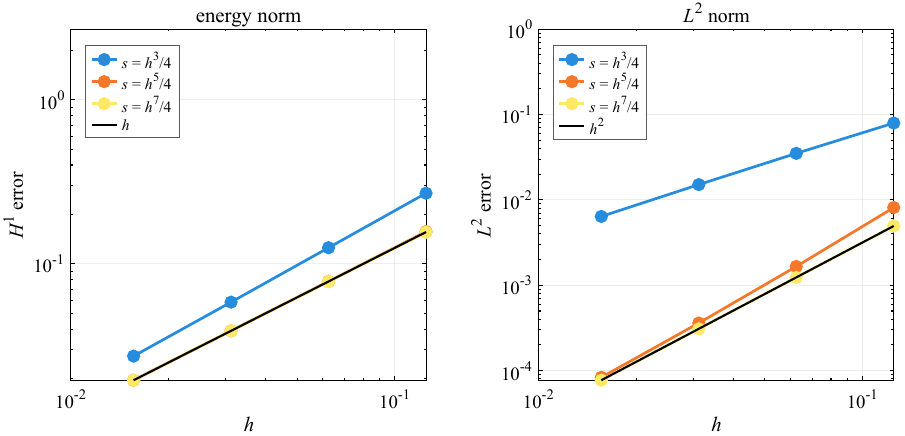}
\caption{Effect of the regularization parameter for $P_1$ elements in
Example A. The sufficient balance $s=h^3/4$ retains the expected energy rate
but pollutes the $L^2(\Omega)$ error. The energy curves for $s=h^5/4$ and
$s=h^7/4$ coincide, showing that the discretization error then dominates.}
\label{fig:svar}
\end{figure}
\FloatBarrier


\subsection{Contact and Parameter Diagnostics}
\label{subsec:contact-parameter-diagnostics}

\paragraph{Contact Pressure and Feasibility.}
The discrete contact pressure
\(\lambda_h=-\gamma\varphi_{+,s}(P(u_h))\) is compared with the exact pressure
\(\sigma_n(u)\) in Figure~\ref{fig:pressure-conv}, using the
\(L^2(\Gamma_C)\) norm and \(s=h^{2k+1}/4\). For Example A, the observed rates
are \(1.10\)--\(1.13\) for \(P_1\) and \(2.03\)--\(2.21\) for \(P_2\), namely
order \(h^k\). The estimate in Remark~\ref{rem:contact-pressure} instead
compares \(\lambda_h\) with the central-path pressure and gives
\(h^{k-1/2}\); the present experiment shows that this half-order loss is not
observed relative to the exact pressure. For Example B with \(P_2\), the rate
decreases from \(2.06\) to \(1.32\) as the mesh resolves the singularity.
Substituting \(u_B\in H^{5/2-\epsilon}(\Omega)\) in the Sobolev estimate of
Remark~\ref{rem:contact-pressure} gives \(h^{1-\epsilon}\), and that estimate
is relative to the central-path pressure rather than \(\sigma_n(u)\). The
values for Example B should therefore be regarded as preasymptotic; they do
not identify a limiting pressure rate.

\begin{figure}[!htbp]
\centering
\includegraphics[width=0.96\textwidth]{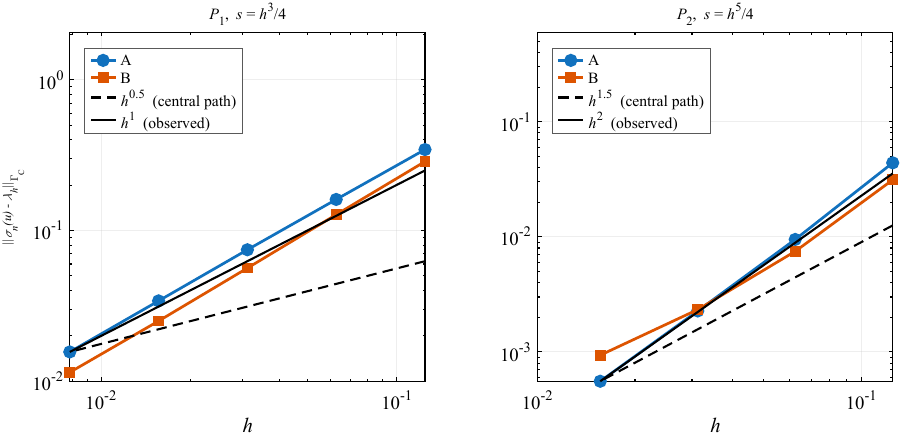}
\caption{Contact-pressure error
$\|\sigma_n(u)-\lambda_h\|_{L^2(\Gamma_C)}$ at $s=h^{2k+1}/4$. The
dashed reference line is the full-regularity order $h^{k-1/2}$ from
Remark~\ref{rem:contact-pressure}; the solid line is the observed order $h^k$
for the smooth problem. Neither line is asserted as an asymptotic rate for
singular Example B.}
\label{fig:pressure-conv}
\end{figure}
\FloatBarrier

In the full-balance exact-solution runs, the largest value of $u_h-g$ over the
contact quadrature points vanished except for Example B with $P_2$, where it
was $4.9\cdot 10^{-5}$ and $2.2\cdot 10^{-5}$ on the two finest meshes. In
the matching $L^2(\Gamma_C)$ norm, the corresponding penetrations were
$4.9\cdot 10^{-6}$ and $1.9\cdot 10^{-6}$, whereas
\begin{align}
\bigl\|\gamma^{-1}\bigl(\sigma_n(u_h)-\lambda_h\bigr)\bigr\|_{\Gamma_C}
=1.0\cdot 10^{-5}
\quad\text{and}\quad
2.7\cdot 10^{-6}
\label{eq:numerical-feasibility-values}
\end{align}
respectively. Thus the observed ratios of the two sides of
\eqref{eq:discrete-feasibility-l2-bound} were $0.48$ and $0.70$, providing a
direct like-norm verification of the discrete feasibility estimate.

\paragraph{The Artificial Gap.}
The central-path relation \eqref{eq:central-path-flux} predicts that too large a value of \(s\)
produces separation rather than penetration, with artificial gap
\(\mu/|\sigma_n(u_\mu)|\). For Example A on the mesh \(n=32\),
Figure~\ref{fig:gap} shows the measure of the computed active set
\(\{P(u_h)>0\}\) and the largest gap \(g-u_h\) over the exact contact set as
\(s\) is swept over
twenty orders of magnitude; the exact contact set has measure one. At
$s = h/4$ the body has lifted off completely, the active set being empty and
the gap $1.3$ for $P_1$ and $1.9$ for $P_2$. The gap then decreases
monotonically, by about five orders of magnitude between $s = h/4$ and
$s = h^7/4$, until it reaches the level of the discretization error,
$8\cdot 10^{-7}$ and $1\cdot 10^{-6}$ respectively. Below that level,
reducing \(s\) changes neither the gap nor the active set, which saturates at
measure \(1.031\) for \(P_1\), one element beyond the exact contact set. The
decrease is slower than linear in \(\mu\) because the maximum is attained near
the free boundary, where \(\sigma_n(u)\) degenerates like \(|x|^3\) and the
predicted gap \(\mu/|\sigma_n(u_\mu)|\) is correspondingly large.

At the sufficient balance $s = h^{2k+1}/4$ the active set measures only
$0.047$ for $P_1$ and $0.62$ for $P_2$, and the artificial gap is
$3.4\cdot 10^{-2}$ and $1.1\cdot 10^{-3}$. Thus this choice preserves the
energy-norm rate, as established in Section~\ref{sec:error-estimates}, while
leaving the contact set and gap poorly resolved. Three or four further powers
of \(h\) are needed before the active set is captured to within one element.
This is
the same phenomenon already seen in the $L^2$ norm in Figure \ref{fig:svar},
and it should be kept in mind when the contact set, rather than the
displacement, is the quantity of interest.

\begin{figure}[!htbp]
\centering
\includegraphics[width=0.96\textwidth]{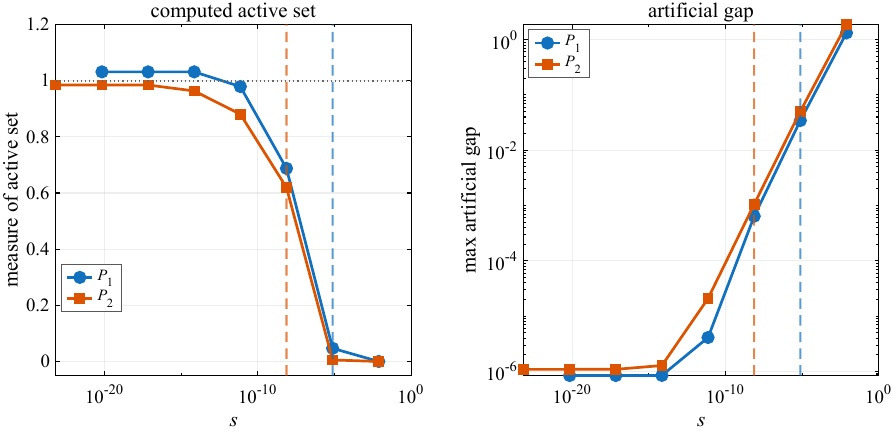}
\caption{Effect of $s$ on the contact set for Example A with $n=32$. Left:
measure of the computed active set $\{P(u_h)>0\}$; the exact measure is one.
Right: maximum artificial gap over the exact contact set. Dashed vertical
lines mark $s=h^{2k+1}/4$ for $P_1$ and $P_2$.}
\label{fig:gap}
\end{figure}
\FloatBarrier

\paragraph{The Nitsche Parameter.}
The theorem requires $\gamma_0$ to be sufficiently large, $\gamma_0 > C_I$
with $C_I$ the constant of the inverse inequality \eqref{eq:inverse}, and the
computations above use $\gamma_0 = 10$ for $P_1$ and $\gamma_0 = 20$ for
$P_2$. Figure \ref{fig:gamma} shows the energy norm error against $\gamma_0$
for Example A on three meshes, with $s$ held at $h^{2k+1}/4$. The observed
stability transition is abrupt and essentially mesh-independent, consistent
with the mesh-independent constant $C_I$: it lies between $\gamma_0 = 1$ and
$\gamma_0 = 1.5$ for $P_1$ and near $\gamma_0 = 5$ for $P_2$. Below it the
iteration either fails to meet the residual criterion or converges to a
discrete solution whose error grows under refinement. For $P_2$, behavior in
the range $1.5 \leq \gamma_0 \leq 5$ is irregular: the coarsest mesh sometimes
succeeds where finer meshes fail.

Above the stability threshold the error grows monotonically with $\gamma_0$,
by an order of magnitude between $\gamma_0 = 1.5$ and $\gamma_0 = 500$, and
asymptotically like $\gamma_0^{1/2}$. This is not a loss of consistency but the
regularization making itself felt: at fixed $s$ the barrier parameter is
$\mu = \gamma s = \gamma_0 s/h$, so raising $\gamma_0$ raises $\mu$ in
proportion, and the central path estimate contributes $\mu^{1/2}$. This is
consistent with the factor $\gamma_0^{-1}$ in the dimensional scaling of
Section~\ref{sec:error-estimates}: if $\gamma_0$ is increased, $s$ must be
decreased in proportion. The values used above
lie safely inside the coercive range, at the price of an error thirty to
eighty per cent above the smallest admissible $\gamma_0$.

\begin{figure}[!htbp]
\centering
\includegraphics[width=0.96\textwidth]{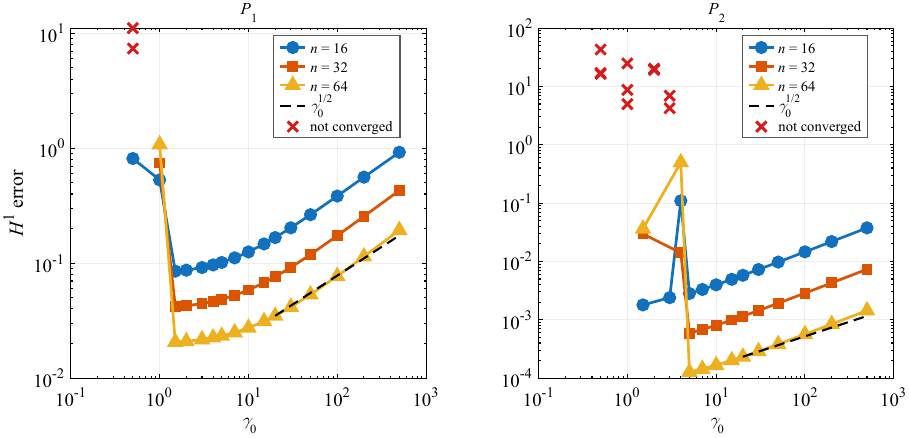}
\caption{Energy error against the Nitsche parameter $\gamma_0$ for Example
A, with $s=h^{2k+1}/4$ fixed. Crosses mark runs that did not satisfy the
residual criterion within $200$ iterations. The dashed reference line has
slope $\gamma_0^{1/2}$, consistent with $\mu=\gamma_0s/h$ and the central-path
estimate.}
\label{fig:gamma}
\end{figure}
\FloatBarrier


\subsection{Mesh Robustness and Free-Boundary Detection}
\label{subsec:mesh-robustness-free-boundary}

\paragraph{Unstructured Meshes.}
All computations so far are on uniform meshes. Since the coercivity rests on
the inverse inequality \eqref{eq:inverse}, whose constant depends on the shape
regularity of the elements, we repeated the convergence study on unstructured
meshes, obtained as Delaunay triangulations of randomly jittered points with
the contact and Dirichlet boundaries resolved at spacing $1/n$; one of them is
shown in the left panel of Figure \ref{fig:unstruct}, and the largest element
diameter is about twice the nominal $1/n$. The observed rates closely match
those on uniform meshes:
$1.08$--$1.09$ for $P_1$ and $2.31$--$2.34$ for $P_2$ in Example A, and for
Example B $1.13$--$1.25$ for $P_1$ and a decrease from $2.34$ to $1.67$ for
$P_2$, showing the same approach to $3/2$ as on the uniform meshes.

The Newton counts remain between $7$ and $12$, and the contact-pressure errors
are similar to those on the uniform meshes. The constraint violation again
vanishes except in the finest $P_2$ computation for Example B. With
$\gamma_0 = 10$ and $20$ the method is thus not close to its coercivity
threshold on meshes of this quality.

\begin{figure}[!htbp]
\centering
\includegraphics[width=0.98\textwidth]{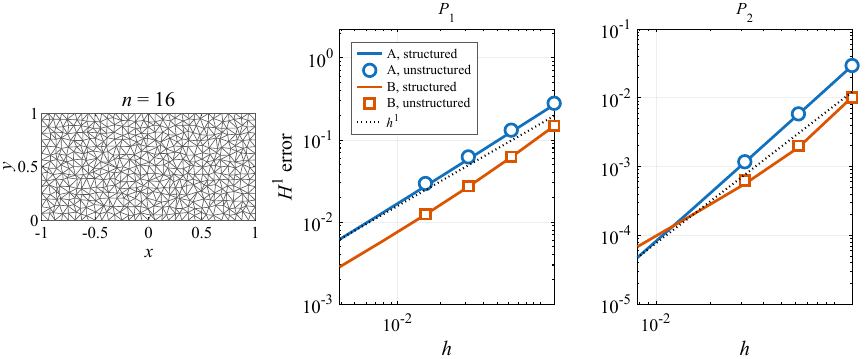}
\caption{Unstructured-mesh study. Left: mesh for $n=16$. Middle and right:
energy errors for Examples A and B, with uniform-mesh results shown as lines
and unstructured-mesh results as open markers. The horizontal coordinate is
the nominal $h=1/n$.}
\label{fig:unstruct}
\end{figure}
\FloatBarrier

\paragraph{Free-Boundary Detection.}
The two examples behave rather differently with respect to the
detection of the free boundary, see Figures \ref{fig:solution-A} and
\ref{fig:solution-B}. For Example B the set of contact edges classified as
active by the sign of $P(u_h)$ coincides with the exact contact set $[-1,0]$
already for $n = 16$. For Example A, on the other hand, the contact pressure
degenerates like $|x|^3$ at the free boundary and therefore falls below the
level of the discretization error well before $x = 0$ is reached; the
computed active set ends at $x = -0.75$, $-0.56$, $-0.38$, $-0.27$ and
$-0.20$ for $n = 8, \dots, 128$. The observed scale is consistent with the
following heuristic. The boundary-norm estimates suggest the scale
$h^{k+1/2}$ for the error in $P(u_h)$. Treating this norm scale pointwise and
balancing it against the exact value $P(u)=|x|^3/\gamma$ gives
$|x| \sim h^{(k-1/2)/3}$, and hence $|x| \sim h^{1/2}$ for $k = 2$, consistent
with the reported values. This balance is not a proved pointwise estimate or
a contact-set convergence result. The energy-norm error is not affected,
precisely because the pressure is negligible in this region, but the
comparison illustrates that a smooth manufactured solution does not
necessarily make the free boundary easier to locate.

\begin{figure}[H]
\centering
\includegraphics[width=0.74\textwidth]{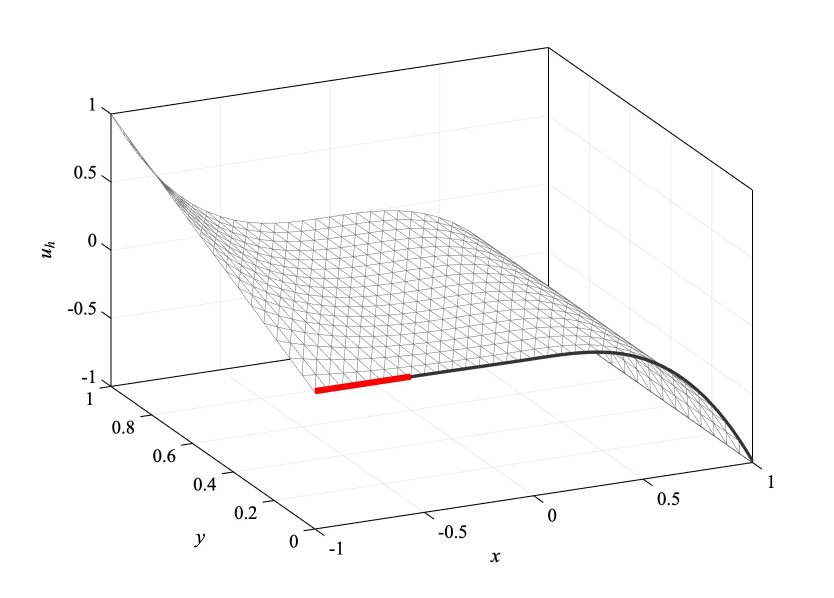}
\caption{Elevation of the $P_2$ solution for Example A ($n=16$,
$s=h^5/4$). Red marks contact edges classified as active and dark gray the
inactive edges. Although the exact contact set is $[-1,0]$, the cubically
degenerate pressure makes the computed active set end near $x=-0.56$ on this
mesh.}
\label{fig:solution-A}
\end{figure}
\FloatBarrier

\paragraph{Reproducibility.}
All parameter sweeps, error evaluations, and figures can be regenerated from
the accompanying MATLAB scripts. The unstructured meshes use a fixed random
seed, and the scripts record the mesh levels, regularization exponents,
iteration counts, and diagnostic norms in the archived data files.
The code and reference datasets are available at
\url{https://github.com/mglarson1/barrier-regularized-signorini}.

\clearpage


\section{Concluding Remarks}
\label{sec:concluding-remarks}

We have introduced a differentiable symmetric Nitsche method for the scalar
Signorini problem and identified its smoothing with a logarithmic barrier on
the augmented-Lagrangian slack variable. Eliminating that variable gives the
perturbed complementarity relation with \(\mu=\gamma s\). This interpretation
is more than algebraic: it identifies a continuous central-path problem with
which the finite element method is exactly consistent.

The resulting analysis separates discretization and regularization. Under the
barrier-feasibility condition of
Section~\ref{subsec:central-path-framework}, the continuous logarithmic energy
has a unique minimizer for every \(\mu>0\), with gap positive almost
everywhere, and Lemma~\ref{lem:central-path-estimate} controls its distance to
the Signorini solution without reciprocal-gap or flux-regularity assumptions.
The Sobolev regularity used for finite element approximation automatically
yields its \(L^2(\Gamma_C)\) Euler equation and normal flux, after which the
discrete solution is quasi-optimal relative to \(u_\mu\), uniformly in \(s\).
The monotonicity of the inverse-gap law also yields an \(H^2\)-bound uniform in
\(\mu\) locally in the interior of every planar contact face, on neighborhoods
that may contain a free-boundary point of the limiting Signorini solution.
Thus the logarithmic boundary layer itself does not cause loss of local
\(H^2\) regularity; the unresolved global obstruction lies at
boundary-condition junctions, corners, and polyhedral edges.
Uniform \(H^r(\Omega)\) regularity gives the sufficient
rate-preserving balance \(s\lesssim h^{2r-1}\); the sufficient choice
\(s\lesssim h^{2k+1}\) is the full-regularity case \(r=k+1\). For each fixed
mesh and \(s>0\), the discrete problem is uniquely solvable, its Newton
matrices are symmetric positive definite for a sufficiently large Nitsche
parameter, and the local Newton rate is quadratic. The same analysis gives
rates for discrete penetration and the complementarity residual.

The computations support the energy-norm estimates and display Newton
behavior consistent with the local theory on both structured and unstructured
meshes. They also reveal an important distinction between displacement error
and contact diagnostics. A smoothing parameter that preserves the energy rate
may still produce a visible artificial gap, degrade the \(L^2(\Omega)\) error,
or identify the contact set poorly. The choice of \(s\) should therefore be
guided by the quantity of interest. Moreover, because
\(\mu=\gamma_0s/h\), increasing \(\gamma_0\) at fixed \(s\) increases the
barrier error and should be accompanied by a proportional reduction of
\(s\).

Existence, uniqueness, and almost-everywhere strict feasibility of the
variational central path and uniform local \(H^2\) regularity on the interiors
of planar contact faces are proved here. Global uniform Sobolev bounds remain
conditional because they can be affected by the Dirichlet--contact interface
and by corner and edge singularities; parameter-uniform regularity above
\(H^2\) near the free boundary of the limiting Signorini solution is also
open. The analysis does not provide
\(L^2(\Omega)\) estimates,
contact-set convergence, or an estimate between the central-path and exact
contact pressures. These are the principal analytical directions suggested by
the numerical results.

\appendix

\section{Uniform Local Regularity Details}
\label{sec:local-regularity-details}

This appendix supplies the uniform estimates and limiting argument used in
Proposition~\ref{prop:uniform-local-central-path-regularity}. Throughout,
\(v_\star\in\mathcal D_B\) is fixed, \(0<\mu\leq\mu_0\), and
\(u_{\mu,\varepsilon}\) denotes the minimizer of
\eqref{eq:regularized-barrier-functional}.

\begin{lem}[Uniform bound for the Moreau problems]
\label{lem:moreau-uniform-h1-bound}
There is \(\varepsilon_0>0\), depending on \(\mu_0\) and the trace and
Poincar\'e constants, such that
\begin{equation}
\sup_{0<\mu\leq\mu_0}\sup_{0<\varepsilon\leq\varepsilon_0}
\|u_{\mu,\varepsilon}\|_{H^1(\Omega)}\leq C_{H^1}
\label{eq:regularized-central-path-uniform-h1-bound}
\end{equation}
The constant depends only on \(\mu_0\), the domain constants,
\(\|f\|_\Omega\), \(\|g\|_{\Gamma_C}\), \(E(v_\star)\), and
\(\mathcal B(v_\star)\), and is independent of \(\mu\) and
\(\varepsilon\).
\end{lem}

\begin{proof}
Let \(C_{\mathrm{tr}}\) be the norm of the trace from \(H^1(\Omega)\) to
\(L^2(\Gamma_C)\). For every \(\delta>0\), the logarithmic lower bound
\eqref{eq:logarithm-quadratic-lower-bound} and minimization of the resulting
quadratic in \(\xi\) give, whenever \(2\delta\varepsilon<1\),
\begin{equation}
\psi_\varepsilon(t)
\geq-\frac{\delta}{1-2\delta\varepsilon}t^2-C_\delta
\quad \forall t\in\mathbb R
\label{eq:moreau-logarithmic-exact-lower-bound}
\end{equation}
Thus, after requiring \(\varepsilon\leq(4\delta)^{-1}\),
\begin{equation}
\psi_\varepsilon(t)\geq-2\delta t^2-C_\delta
\quad \forall t\in\mathbb R
\label{eq:moreau-logarithmic-lower-bound}
\end{equation}
Since \(\psi_\varepsilon\leq\psi\), minimality and comparison with the fixed
barrier-feasible function give
\begin{equation}
J_{\mu,\varepsilon}(u_{\mu,\varepsilon})
\leq J_{\mu,\varepsilon}(v_\star)
\leq E(v_\star)+\mu_0\max\{\mathcal B(v_\star),0\}
\label{eq:moreau-comparison-upper-bound}
\end{equation}
On the other hand, \eqref{eq:moreau-logarithmic-lower-bound}, the trace
inequality, and \(a(v,v)\geq c_P\|v\|_{H^1(\Omega)}^2\) imply
\begin{align}
J_{\mu,\varepsilon}(v)
&\geq
\left(\frac{c_P}{2}-4\mu_0\delta C_{\mathrm{tr}}^2\right)
\|v\|_{H^1(\Omega)}^2
-\|f\|_\Omega\|v\|_{H^1(\Omega)}
\label{eq:moreau-uniform-coercivity-first}\\
&\quad
-4\mu_0\delta\|g\|_{\Gamma_C}^2
-\mu_0 C_\delta\lvert\Gamma_C\rvert
\label{eq:moreau-uniform-coercivity-data}
\end{align}
Choose \(\delta\leq c_P/(16\mu_0C_{\mathrm{tr}}^2)\) and then
\(\varepsilon_0\leq(4\delta)^{-1}\). Young's inequality yields
\begin{equation}
J_{\mu,\varepsilon}(v)
\geq\frac{c_P}{8}\|v\|_{H^1(\Omega)}^2-C_{\mathrm{data}}
\label{eq:moreau-uniform-coercivity-final}
\end{equation}
uniformly in the stated parameter ranges. Combining
\eqref{eq:moreau-comparison-upper-bound} and
\eqref{eq:moreau-uniform-coercivity-final} proves
\eqref{eq:regularized-central-path-uniform-h1-bound}.
\end{proof}

\begin{lem}[Local estimate for a monotone boundary law]
\label{lem:local-monotone-boundary-h2}
Under the geometric and data assumptions of
Proposition~\ref{prop:uniform-local-central-path-regularity},
\begin{equation}
\|u_{\mu,\varepsilon}\|_{H^2(U_0\cap\Omega)}
\leq C_{\mathrm{loc}}
\label{eq:regularized-central-path-local-h2-bound}
\end{equation}
for \(0<\mu\leq\mu_0\) and
\(0<\varepsilon\leq\varepsilon_0\). The constant is independent of both
parameters.
\end{lem}

\begin{proof}
The compact part \(\overline U_0\cap\partial\Omega\) is covered by finitely
many nested flat half-cylinders \(Q_0^+\Subset Q_1^+\Subset U_1\cap\Omega\).
After a rigid change of coordinates, the flat boundary is
\(\Sigma_1=Q_1\cap\{x_d=0\}\), while
\(Q_1^+=Q_1\cap\{x_d>0\}\). It suffices to prove the estimate on one such
pair. Choose \(\eta\in C_c^\infty(Q_1)\) equal to one on \(Q_0\), use the
local \(H^2\)-extension of the obstacle, and set
\begin{equation}
z\coloneqq u_{\mu,\varepsilon}-G,
\qquad
\beta_{\mu,\varepsilon}\coloneqq\mu\psi_\varepsilon'
\label{eq:local-regularized-gap-and-law}
\end{equation}
For test functions supported in \(Q_1^+\cup\Sigma_1\), the regularized Euler
equation becomes
\begin{equation}
(\nabla z,\nabla w)_{Q_1^+}
+(\beta_{\mu,\varepsilon}(z),w)_{\Sigma_1}
=(f,w)_{Q_1^+}-(\nabla G,\nabla w)_{Q_1^+}
\label{eq:localized-regularized-euler-equation}
\end{equation}

Let \(\boldsymbol\tau\) be any unit direction tangent to \(\Sigma_1\) and
define
\begin{equation}
\delta_hv(x)\coloneqq
\frac{v(x+h\boldsymbol\tau)-v(x)}{h}
\label{eq:tangential-difference-quotient}
\end{equation}
Tangential translations preserve the half-cylinder. For functions whose
support stays inside the translated patch, discrete integration by parts gives
\begin{equation}
(a,-\delta_{-h}b)_D=(\delta_ha,b)_D
\label{eq:discrete-integration-by-parts}
\end{equation}
both for \(D=Q_1^+\) and for \(D=\Sigma_1\). Set
\(W_h=\eta^2\delta_hz\) and take
\(w=-\delta_{-h}W_h\) in
\eqref{eq:localized-regularized-euler-equation}. This is an admissible
\(V\)-test function because the support of \(\eta\) is separated from the
edge of the contact face. Expanding \(\nabla W_h\) gives
\begin{align}
\|\eta\delta_h\nabla z\|_{Q_1^+}^2
&+2(\eta\delta_h\nabla z,\delta_hz\nabla\eta)_{Q_1^+}
+(\delta_h\beta_{\mu,\varepsilon}(z),
\eta^2\delta_hz)_{\Sigma_1}
\label{eq:localized-difference-quotient-identity-left}\\
&=(f,-\delta_{-h}W_h)_{Q_1^+}
-(\delta_h\nabla G,\nabla W_h)_{Q_1^+}
\label{eq:localized-difference-quotient-identity-right}
\end{align}
The boundary term is nonnegative pointwise because
\begin{equation}
\delta_h\beta_{\mu,\varepsilon}(z)\,\delta_hz
=\frac{
[\beta_{\mu,\varepsilon}(z(x+h\boldsymbol\tau))
-\beta_{\mu,\varepsilon}(z(x))]
[z(x+h\boldsymbol\tau)-z(x)]}{h^2}
\geq0
\label{eq:pointwise-boundary-monotonicity}
\end{equation}

The usual difference-quotient estimate, applied on a slightly enlarged
support inside \(Q_1\), yields
\begin{equation}
\|\delta_hz\|_{Q_1^+}\leq C\|\nabla z\|_{Q_1^+}
\label{eq:first-difference-quotient-bound}
\end{equation}
Moreover,
\begin{equation}
\|\delta_{-h}W_h\|_{Q_1^+}
+\|\nabla W_h\|_{Q_1^+}
\leq C\bigl(
\|\eta\delta_h\nabla z\|_{Q_1^+}
+\|z\|_{H^1(Q_1^+)}\bigr)
\label{eq:localized-test-function-bound}
\end{equation}
and \(\|\delta_h\nabla G\|_{Q_1^+}\leq C\|G\|_{H^2(Q_1^+)}\).
Consequently, Cauchy--Schwarz and Young's inequality in
\eqref{eq:localized-difference-quotient-identity-left}--
\eqref{eq:localized-difference-quotient-identity-right}, together with the
nonnegative boundary term, give
\begin{equation}
\|\eta\delta_h\nabla z\|_{Q_1^+}^2
\leq C\left(
\|z\|_{H^1(Q_1^+)}^2
+\|f\|_{Q_1^+}^2
+\|G\|_{H^2(Q_1^+)}^2\right)
\label{eq:local-tangential-difference-quotient-bound}
\end{equation}
The right-hand side contains no derivative of \(f\). By
Lemma~\ref{lem:moreau-uniform-h1-bound}, its constant is uniform in
\(\mu\) and \(\varepsilon\). Letting \(h\to0\) controls every second
derivative containing a tangential derivative.

Tests compactly supported in \(Q_1^+\) show that
\(-\Delta u_{\mu,\varepsilon}=f\). In the flat coordinates this gives
\begin{equation}
\partial_{dd}z
=-f-\Delta G-\sum_{j=1}^{d-1}\partial_{jj}z
\label{eq:local-central-path-normal-second-derivative}
\end{equation}
which controls the remaining second derivative. A finite covering of
\(\overline U_0\cap\partial\Omega\), together with the standard interior
estimate away from the boundary, proves
\eqref{eq:regularized-central-path-local-h2-bound}.
\end{proof}

\begin{lem}[Limit of the Moreau problems]
\label{lem:moreau-central-path-limit}
For every fixed \(\mu>0\),
\begin{equation}
u_{\mu,\varepsilon}\to u_\mu
\quad\text{strongly in }H^1(\Omega)
\quad\text{as }\varepsilon\downarrow0
\label{eq:moreau-minimizer-strong-convergence}
\end{equation}
The uniform local estimate
\eqref{eq:regularized-central-path-local-h2-bound} therefore passes to
\(u_\mu\), with its constant still uniform for \(0<\mu\leq\mu_0\).
\end{lem}

\begin{proof}
The Moreau envelopes increase pointwise to \(\psi\) as
\(\varepsilon\downarrow0\). For completeness, their Fenchel representation
is
\begin{equation}
\psi_\varepsilon(t)
=\sup_{p\in\mathbb R}
\left(pt-\psi^*(p)-\frac{\varepsilon p^2}{2}\right)
\label{eq:moreau-fenchel-representation}
\end{equation}
Hence, if \(t_j\to t\) and \(\varepsilon_j\downarrow0\), then
\begin{equation}
\liminf_{j\to\infty}\psi_{\varepsilon_j}(t_j)\geq\psi(t)
\label{eq:moreau-pointwise-liminf}
\end{equation}
Indeed, the liminf is bounded below by
\(pt-\psi^*(p)\) for every fixed \(p\), and taking the supremum proves the
claim.

Suppose that \(v_j\rightharpoonup v\) in \(H^1(\Omega)\). Compactness of the
trace gives, after passage to a subsequence,
\(v_j-g\to v-g\) strongly in \(L^2(\Gamma_C)\) and almost everywhere.
Adding the uniform quadratic lower-bound correction from
\eqref{eq:moreau-logarithmic-lower-bound} and applying Fatou's lemma to
\eqref{eq:moreau-pointwise-liminf} gives
\begin{equation}
\int_{\Gamma_C}\psi(v-g)\,ds
\leq\liminf_{j\to\infty}
\int_{\Gamma_C}\psi_{\varepsilon_j}(v_j-g)\,ds
\label{eq:moreau-integral-liminf}
\end{equation}
For a fixed \(v\in\mathcal D_B\), the same quadratic correction and monotone
convergence give
\begin{equation}
\int_{\Gamma_C}\psi_\varepsilon(v-g)\,ds
\to\mathcal B(v)
\quad\text{as }\varepsilon\downarrow0
\label{eq:moreau-integral-recovery}
\end{equation}
Equations~\eqref{eq:moreau-integral-liminf} and
\eqref{eq:moreau-integral-recovery} are the Mosco liminf and recovery
properties for \(J_{\mu,\varepsilon}\).

Lemma~\ref{lem:moreau-uniform-h1-bound} gives weak compactness of the
minimizers. Every weak cluster point minimizes \(J_\mu\), by the liminf
property, minimality of \(u_{\mu,\varepsilon}\), and recovery with
\(v=u_\mu\). The minimizer of \(J_\mu\) is unique, so the whole family
converges weakly to \(u_\mu\), and the minimum values converge. Strong
convexity then yields
\begin{equation}
\frac{c_P}{2}
\|u_{\mu,\varepsilon}-u_\mu\|_{H^1(\Omega)}^2
\leq J_{\mu,\varepsilon}(u_\mu)
-J_{\mu,\varepsilon}(u_{\mu,\varepsilon})
\to0
\label{eq:moreau-strong-convexity-convergence}
\end{equation}
which proves \eqref{eq:moreau-minimizer-strong-convergence}.

Finally, Lemma~\ref{lem:local-monotone-boundary-h2} gives weak compactness in
\(H^2(U_0\cap\Omega)\). Its weak limit agrees with \(u_\mu\) by the strong
\(H^1\) convergence, and weak lower semicontinuity transfers the uniform
local bound to \(u_\mu\).
\end{proof}

\paragraph{Acknowledgments.}
PH was supported by the Swedish Research Council under grant 2022-03908. MGL
was supported in part by the Swedish Research Council under grants 2021-04925
and 2025-05562, the Knut and Alice Wallenberg Foundation under grant
KAW 2025.0277, and the Swedish Research Programme Essence.

\paragraph{Use of Artificial Intelligence and Computational Tools.}
During the preparation and revision of this manuscript, the authors used
OpenAI's GPT-5.6 Sol through ChatGPT and Codex, as well as Claude Fable 5 and
Claude Opus 5, for language editing, literature searches, and manuscript
refinement. MATLAB was used for numerical computations and figure preparation.
The authors reviewed all resulting material and take full responsibility for
the content of the manuscript.

\enlargethispage{\baselineskip}
\paragraph{Authors' Addresses.}\mbox{}\par
\begingroup
\fontsize{9}{10.5}\selectfont
\noindent
Peter Hansbo, Department of Mechanical Engineering, J\"onk\"oping University,
Sweden,\\
\hspace*{1em}\texttt{peter.hansbo@ju.se}.\par
\noindent
Mats G. Larson, Department of Mathematics and Mathematical Statistics,
Ume{\aa} University, Sweden,\\
\hspace*{1em}\texttt{mats.larson@umu.se}.
\endgroup

\clearpage

\footnotesize{

}

\end{document}